\documentclass[11pt]{amsart}
\usepackage{amsfonts}
\usepackage{}

\usepackage{amsmath}
\usepackage{amssymb}
\usepackage{latexsym}
\usepackage{amscd}
\usepackage{mathrsfs}
\usepackage[all]{xy}

\newdimen\AAdi%
\newbox\AAbo%
\def\AAk#1#2{\s_etbox\AAbo=\hbox{#2}\AAdi=\wd\AAbo\kern#1\AAdi{}}%
\def\AAr#1#2#3{\s_etbox\AAbo=\hbox{#2}\AAdi=\ht\AAbo\raise#1\AAdi\hbox{#3}}%
\font\tenmsb=msbm10 at 12pt \font\sevenmsb=msbm7 at 8pt
\font\fivemsb=msbm5 at 6pt
\newfam\msbfam
\textfont\msbfam=\tenmsb \scriptfont\msbfam=\sevenmsb
\scriptscriptfont\msbfam=\fivemsb

\newtheorem{theorem}{Theorem}

\newtheorem{remark}[theorem]{Remark}

\newtheorem{corollary}[theorem]{Corollary}

\newtheorem{lemma}[theorem]{Lemma}
\newtheorem{proposition}[theorem]{Proposition}

\numberwithin{equation}{section} \numberwithin{theorem}{section}

\renewcommand{\topmargin}{0cm}
\renewcommand{\oddsidemargin}{5mm}
\renewcommand{\evensidemargin}{5mm}
\renewcommand{\textwidth}{150mm}
\renewcommand{\textheight}{230mm}

\def\R{\mathbb R}

\def\Z{\mathbb Z}

\def\Z{\mathbb Z}

\def\na{\nabla}
\def\bn{\overline\nabla}

\def\f#1#2{\frac{#1}{#2}}

\def\a{\alpha}
\def\be{\beta}

\def\r{\Re_{I\!V}}

\def\p#1{\partial #1}

\def\de{\delta}
\def\De{\Delta}
\def\e{\eta}
\def\ep{\epsilon}

\def\G{\Gamma}

\def\k{\kappa}
\def\la{\lambda}
\def\La{\Lambda}
\def\lan{\langle}
\def\ran{\rangle}

\def\Om{\Omega}
\def\th{\theta}
\def\Th{\Theta}

\def\Si{\Sigma}

\def\r{\rho}

\begin{document}

\title
[Liouville type theorems and Hessian estimates]
{Liouville type theorems and Hessian estimates for special Lagrangian equations}
\author{Qi Ding}
\address{Shanghai Center for Mathematical Sciences, Fudan University, Shanghai 200438, China}
\email{dingqi@fudan.edu.cn}

\thanks{
The author would like to express his sincere gratitude to the referees for valuable comments that will help me to improve the quality of the manuscript.
The author is partially supported by NSFC 11871156 and NSFC 11922106. }

\begin{abstract}

In this paper, we get a Liouville type theorem for the special
Lagrangian equation with a certain 'convexity' condition, where Warren-Yuan first studied the condition in \cite{WmY2}.
Based on Warren-Yuan's work, our strategy is to show a global Hessian estimate of solutions via the Neumann-Poincar$\mathrm{\acute{e}}$ inequality on special Lagrangian graphs, and mean value inequality for superharmonic functions on these graphs, where we need geometric measure theory.
Moreover, we derive interior Hessian estimates on the gradient of the solutions to the equation with this 'convexity' condition or with supercritical phase.

\end{abstract}

\maketitle

\section{Introduction}

Let $u$ be a smooth function on an open set $\Om\subset\R^n$, then $M\triangleq\{(x,Du(x))\in\R^n\times\R^n|\, x\in\Om\}$ is a Lagrangian submanifold in $\R^n\times\R^n$.
Let $\la_1(x),\cdots,\la_n(x)$ be the eigenvalues of the Hessian matrix $D^2u(x)$ at any point $x\in\Om$.
We call $M$ a \emph{special Lagrangian graph} if $u$ is a solution to the special Lagrangian equation
\begin{equation}\aligned\label{SL}
\sum_{i=1}^n\arctan\la_i=\Th,\qquad\qquad \mathrm{for\ some\ constant\ } \Th.
\endaligned
\end{equation}
The equation \eqref{SL} arises in the special Lagrangian geometry by Harvey-Lawson \cite{HL}.
$M$ is the special Lagrangian graph if and only if $M$ is a minimal submanifold in $\R^n\times\R^n$, or the calibrating $n$-form
$Re(e^{-\sqrt{-1}\Th}dz_1\wedge\cdots\wedge dz_n)$ is equal to the induced volume form along $M$, which is also equivalent to that $M$ is (volume)
minimizing in $\R^n\times\R^n$ (see Theorem 2.3, Proposition 2.17 in  \cite{HL}; or Chapter 5 in \cite{X}).

The classification of global solutions to \eqref{SL} on $\R^n$ has a long history.
In 1998, Fu \cite{F} classified any smooth solution to \eqref{SL} on $\R^2$, i.e.,
any such solution is either quadratic for $|\Th|>0$ or harmonic for $\Th=0$. In particular, \eqref{SL} for $\Th=\f\pi2$ is just the Monge-Amp$\mathrm{\grave{e}}$re equation of dimension 2. Let $u$ be a smooth solution to \eqref{SL} on $\R^n$.
In high dimensions, Yuan \cite{Y2} proved that $u$ must be quadratic for $|\Th|>\f{n-2}2\pi$.
For $\Th=k\pi$ with integer $k$, Borisenko \cite{Bo} proved that $u$ is affine provided $u$ has the linear growth.
For general $n$, Jost-Xin \cite{JX} showed that the convex solution $u$ is quadratic provided the Hessian $D^2u$ is uniformly bounded.
For $n=3$ and $\Th=\pi$, Bao-Chen-Guan-Ji \cite{BCGJ} proved that the strictly convex $u$ with quadratic growth must be quadratic.
Using Lewy rotation brilliantly, Yuan \cite{Y1} proved that the convex solution $u$ must be quadratic for each $n$.

Furthermore, the Liouville theorem may hold true under conditions much weaker than convexity.
Let $u$ be a smooth solution to \eqref{SL} on $\R^n$ with the eigenvalues $\la_1,\cdots,\la_n$ of the Hessian matrix $D^2u$.
In the same paper \cite{Y1}, Yuan proved the existence of the constant $\ep'>0$ depending only on $n$ such that $u$ is quadratic provided $D^2u\ge-\ep'$ on $\R^n$. Further, for $n=3$ Yuan proved that $u$ is quadratic on $\R^3$ if $D^2u$ is uniformly bounded from below \cite{Y1}, or $\la_i\la_j$ is uniformly bounded from below for all $i,j$ \cite{Y3}.
Moreover, Tsui-Wang \cite{TW} proved that if $\la_i\la_j\ge-\f32+\tau$ for all $i,j$ and any fixed constant $\tau>0$, and $|D^2u|$ is uniformly bounded, then $u$ is quadratic.
In \cite{Y3}, Yuan proved that $u$ is quadratic if one of following statement holds: (i) $\la_i\ge-\f1{\sqrt{3}}+\de$ everywhere for every $i,j$ and any fixed constant $\de>0$ (or 'equivalently' $|\la_i|\le\sqrt{3}-\de'$ for every $i$ and any fixed constant $\de'>0$); (ii) $\la_i\la_j\ge-1-\de''$ everywhere for every $i,j$ and any fixed constant $\de''>0$.
In \cite{WmY2}, Warren-Yuan first introduced a more general 'convexity' condition:
\begin{equation}\aligned\label{3eplailaj}
3+(1-\ep)\la_i^2(x)+2\la_i(x)\la_j(x)\ge0
\endaligned
\end{equation}
for all $i,j,x$ and any small fixed $\ep>0$,
which appeared naturally in studying subharmonicity of $\log\det(I+D^2uD^2u)$ on the special Lagrangian graph of the graphic function $Du$.
Under the condition \eqref{3eplailaj} and $|Du|<\de(n)|x|$ for large $|x|$ and any fixed constant $\de(n)<\f1{\sqrt{n-1}}$, Warren-Yuan showed that $u$ is quadratic \cite{WmY2}.
Moreover, they also proved that $u$ is quadratic provided \eqref{3eplailaj} holds for $\ep=0$ and $ D^2u$ is uniformly bounded on $\R^n$.

In this paper, we show a Liouville type theorem for special Lagrangian graphs under the condition \eqref{3eplailaj} for $\ep=0$.
\begin{theorem}\label{MAIN}
Let $u$ be a smooth solution to the special Lagrangian equation \eqref{SL} on $\R^n$,
where $\la_1(x),\cdots,\la_n(x)$ are the eigenvalues of the Hessian $D^2u(x)$. If
\begin{equation}\aligned\label{3lailaj}
3+\la_i^2(x)+2\la_i(x)\la_j(x)\ge0
\endaligned
\end{equation}
holds
for all $i,j=1,\cdots,n$ and $x\in\R^n$, then $u$ must be a quadratic polynomial.
\end{theorem}
In fact, we have a litter stronger result than the above theorem. More precisely, there exists a constant $\ep_n>0$ depending only on $n$ such that if a smooth solution $u$ to \eqref{SL} on $\R^n$ satisfies
$$3(1+\ep_n)+(1+\ep_n)\la_i^2+2\la_i\la_j\ge0$$
on $\R^n$ for all $i,j$, then $u$ is a quadratic polynomial (see Theorem \ref{mainLiou}).
Using Warren-Yuan's argument in \cite{WmY2}, in order to prove Theorem \ref{MAIN} it is sufficient to show the following global Hessian estimate.
\begin{theorem}\label{GLOBALBDu}
For any constant $K\ge1$, there is a constant $c_{n,K}>0$ depending only on $n,K$ such that if $u$ is a smooth solution to \eqref{SL} on $\R^n$ with the eigenvalues $\la_1,\cdots,\la_n$ of the Hessian $D^2u$ satisfying
\begin{equation}\aligned\label{Assump0}
\la_i\la_j\ge-K\qquad\qquad on\ \R^n
\endaligned
\end{equation}
for all $i,j=1,\cdots,n$, and $u$ is not a quadratic polynomial, then the Hessian of $u$ satisfies $-c_{n,K}\le D^2u\le c_{n,K}$ on $\R^n$.
\end{theorem}

The geometric meaning of \eqref{Assump0} is that determinant of $\mathrm{Hess}_Su$ on any 2-dimensional surface $S$ of $\R^n$ has a lower bound by $-K$,
where $\mathrm{Hess}_Su$ is the Hessian of $u$ restricted on $S$.
Without the condition \eqref{Assump0}, $D^2u$ may be unbounded. For instance, those harmonic functions have the unbounded Hessian on $\R^2$ as they are solutions to \eqref{SL} for $n=2$, $\Th=0$.
Theorem \ref{GLOBALBDu} is proved by contradiction with the help of geometric measure theory, where we use the mean value inequality on special Lagrangian graphs for superharmonic functions in terms to the Hessian of solutions. Here, the mean value inequality is established due to
the Neumann-Poincar$\mathrm{\acute{e}}$ inequality on the graphs.
It is worth to point out that Bombieri-Giusti had established the Neumann-Poincar$\mathrm{\acute{e}}$ inequality on area-minimizing hypersurfaces in Euclidean space,
and given many applications to area-minimizing hypersurfaces \cite{BG}.

One application of Theorem \ref{MAIN} is the interior curvature estimate of special Lagrangian graphs (see Corollary \ref{CurvEst}).
With curvature estimate, we can obtain a new interior Hessian estimate for solutions of the special Lagrangian equation \eqref{SL}. Before stating our result,
let us review the known results in this direction.

In the 1950s, Heinz derived a Hessian bound for \eqref{SL} with $n=2$ and $\Th=\pi/2$ (i.e., the Monge-Amp$\mathrm{\grave{e}}$re equation); Pogorelov \cite{P} got Hessian estimates for (1.1) with $n=2$ and $\Th>\pi/2$.
Bao-Chen \cite{BC} got Hessian estimates in terms of certain integrals of the Hessian
for solutions to \eqref{SL} with $n=3$, $\Th=\pi$.
Warren-Yuan obtained Hessian estimates of \eqref{SL} in terms of gradients for solutions to \eqref{SL} in the following cases:
i) the solutions satisfies \eqref{3eplailaj} with small gradients in \cite{WmY2}; ii) $n=2$ in \cite{WmY4}; iii) $n=3$ and $|\Th|\ge\f{\pi}2$ in \cite{WmY3,WmY1}.
For general $n$, Chen-Warren-Yuan \cite{CWY} derived a priori interior Hessian estimates for smooth convex solutions to \eqref{SL}
(see the very recent work \cite{CSY} for convex viscosity solutions).
In \cite{WdY}, Wang-Yuan obtained a priori interior Hessian estimates for all the solutions to \eqref{SL} with critical and supercritical phases in dimensions $\ge3$.
More precisely, for any $n\ge3$, there is a constant $c_n$ depending on $n$ such that for any smooth solution on $B_R(0)\subset\R^n$ to \eqref{SL} with $|\Th|\ge(n-2)\f{\pi}2$, there holds
\begin{equation}\aligned
|D^2u(0)|\le c_{n} \mathrm{exp}\left(c_{n}R^{2-2n}\max_{B_R(0)}|Du|^{2n-2}\right),
\endaligned
\end{equation}
and when $|\Th|=(n-2)\f{\pi}2$, there holds
\begin{equation}\aligned
|D^2u(0)|\le c_{n} \mathrm{exp}\left(c_{n}R^{4-2n}\max_{B_R(0)}|Du|^{2n-4}\right).
\endaligned
\end{equation}
From the counter-examples constructed by Nadirashvili-Vl$\mathrm{\breve{a}}$dut \cite{NV} and Wang-Yuan \cite{WdY0},
the condition $|\Th|\ge(n-2)\f{\pi}2$ above is necessary.

Hessian estimates for the special Lagrangian equation \eqref{SL} are equivalent to gradient estimates for special Lagrangian graphs, which are minimizing submanifolds.
In \cite{Fi}, Finn obtained gradient estimates in terms of the linear exponential dependence on the solutions of 2-dimensional minimal surfaces equation.
In high dimensions, Bombieri-De Giorgi-Miranda \cite{BDM} derived gradient estimates in terms of the linear exponential dependence on the solutions to
minimal hypersurfaces equation. Wang studied the high codimension case under some conditions in \cite{W1}. 

With curvature estimate and the mean value inequality on special Lagrangian graphs for superharmonic functions,
we derive a new interior Hessian estimate in terms of the exponential dependence on the $n$-th power of gradient of the solutions.
\begin{theorem}\label{Hessest0}
Let $u$ be a smooth solution to the special Lagrangian equation \eqref{SL} on $B_R(0)\subset\R^n$. Suppose that \eqref{3lailaj} holds on $B_R(0)$
for all $i,j=1,\cdots,n$. Then there is a constant $C_{n}>0$ depending only on $n$ such that
\begin{equation}\aligned\label{Sharpest}
|D^2u(0)|\le C_{n} \mathrm{exp}\left(C_{n}\f{\max_{B_R(0)}|Du-Du(0)|^n}{R^n}\right).
\endaligned
\end{equation}
\end{theorem}
In \cite{Fi}, Finn constructed solutions to minimal surface equation in $\R^3$ whose gradients have the linear exponential dependence on the solutions.
With Heinz transformation \cite{J}, there is a smooth solution $\psi$ to \eqref{SL} with $n=2$ and $\Th=\pi/2$ (i.e., the Monge-Amp$\mathrm{\grave{e}}$re equation), whose Hessian indeed has the linear exponential dependence on $D\psi$ (see also the introduction in \cite{WdY}). Here, the order $n$ in \eqref{Sharpest} is just the codimension of the special Lagrangian graph $\{(x,Du(x))\in\R^n\times\R^n|\, x\in B_R(0)\}$ in $\R^n\times\R^n$.

Our Hessian estimate \eqref{Sharpest} is effective for the smooth convex solutions. Moreover, via subharmonic functions obtained by Wang-Yuan \cite{WdY}, our strategy of the proof of Theorem \ref{Hessest0} is effective for smooth solutions to the special Lagrangian equation \eqref{SL} with supercritical phase, i.e., $|\Th|>(n-2)\pi/2$ (see Theorem \ref{GEsptP}).
However, the strategy is ineffective for critical phase, i.e., $|\Th|=(n-2)\pi/2$, because in this situation the Hessian of the solutions may be not uniformly bounded from below for $\Th=(n-2)\pi/2$, or above for $\Th=-(n-2)\pi/2$.

\section{Lewy rotation for special Lagrangian graphs over convex sets}

In this paper, we denote $B_r(x)$ be the ball in $\R^n$ with the radius $r$ and centered at $x\in\R^n$.
Denote $\mathbf{B}_r(\mathbf{x})$ be the ball in $\R^{n+n}$ with the radius $r$ and centered at $\mathbf{x}\in\R^{n+n}$.
Let $B_r=B_r(0)$, $\mathbf{B}_r=\mathbf{B}_r(\mathbf{0})$ for convenience.
For any subset $E$ in $\R^n$ and any constant $0\le s\le n$, let $\mathcal{H}^s(E)$ denote the $s$-dimensional Hausdorff measure of $E$.
Let $\Pi$ be a projection from $\R^n\times\R^n$ into $\R^n$ defined by $\Pi(\mathbf{x})=x$ for any $\mathbf{x}=(x,y)\in\R^n\times\R^n$.

Let $u$ be a smooth solution to the special Lagrangian equation \eqref{SL} on an open convex set $\Om$ of $\R^n$.
Assume
\begin{equation}\aligned\label{lowerLa}
\inf_{\Om}D^2u\ge-\La\qquad\qquad\mathrm{for\ some\ constant}\ \La>0.
\endaligned
\end{equation}
Then $\hat{u}(x)\triangleq u(x)+\f{\La}2|x|^2$ is convex. For $x,x'\in\Om$, from the segment $\overline{xx'}\subset\Om$ we have
\begin{equation}\aligned
\lan x-x',D\hat{u}(x)-D\hat{u}(x')\ran\ge0,
\endaligned
\end{equation}
which implies
\begin{equation}\aligned\label{x2x1}
\lan x-x',Du(x)-Du(x')\ran\ge-\La|x-x'|^2.
\endaligned
\end{equation}
In \cite{Y1}, Yuan introduced the Lewy rotation as follows, which turns out to be a standard technique nowadays, but still very powerful in studying special Lagrangian equation.
Let $F_\La:\,(x,y)\rightarrow(\hat{x},\hat{y})$ be the Lewy rotation defined by
\begin{equation}\label{labxby}
\left\{\begin{split}
&\hat{x}=(\hat{x}_1,\cdots,\hat{x}_n)=\f1{\sqrt{4\La^2+1}}(2\La x+y)\\
&\hat{y}=(\hat{y}_1,\cdots,\hat{y}_n)=\f1{\sqrt{4\La^2+1}}(-x+2\La y)\\
\end{split}\right.,
\end{equation}
which is an isometry from $\R^n\times\R^n$ to $\R^n\times\R^n$.
Let $M$ be a graph over $\Om$ defined by $\{(x,Du(x))\in\R^n\times\R^n|\ x\in\Om\}$.
We call $M$ \emph{a special Lagrangian graph}.
Let $\bar{x},\bar{y}:\ \Om\to\R^n$ be smooth mappings defined by
\begin{equation}\aligned\label{barxbaryx}
\bar{x}(x)=&\hat{x}(x,Du(x))=\f1{\sqrt{4\La^2+1}}(2\La x+Du(x))\\
\bar{y}(x)=&\hat{y}(x,Du(x))=\f1{\sqrt{4\La^2+1}}(-x+2\La Du(x))
\endaligned\qquad\mathrm{for\ any}\ x\in\Om.
\end{equation}

Combining \eqref{x2x1} (see also \cite{Y1}), for any $x,x'\in\Om$ we have
\begin{equation}\aligned\label{Dux2x1}
&|\bar{x}(x)-\bar{x}(x')|^2=|\hat{x}(x,Du(x))-\hat{x}(x',Du(x'))|^2\\
=&\f1{4\La^2+1}\Big(4\La^2|x-x'|^2+4\La\lan x-x',Du(x)-Du(x')\ran+|Du(x)-Du(x')|^2\Big)\\
\ge&\f1{4\La^2+1}\Big(2\La^2|x-x'|^2+2\La\lan x-x',Du(x)-Du(x')\ran+|Du(x)-Du(x')|^2\Big)\\
\ge&\f{\La^2}{4\La^2+1}|x-x'|^2.
\endaligned
\end{equation}
Hence, $\bar{x}:\, \Om\rightarrow \bar{x}(\Om)=\{\bar{x}(x,Du(x))|\,x\in\Om\}$ is injective and then $F_\La(M)$ is a graph over $\bar{x}(\Om)$.

Let $J$ be the Jacobi of the mapping $\bar{x}$, i.e., 
\begin{equation}\aligned\label{DEFJ}
J=\left(\f{\p\bar{x}_i}{\p x_j}\right)=\f1{\sqrt{4\La^2+1}}(2\La I+D^2u(x)),
\endaligned
\end{equation}
and $\widehat{J}$ be the Jacobi of the mapping $\bar{y}$, i.e., $$\widehat{J}=\left(\f{\p\bar{y}_i}{\p x_j}\right)=\f1{\sqrt{4\La^2+1}}(-I+2\La D^2u(x)).$$
Note that both of $J$ and $\widehat{J}$ are symmetric matrices. With the diagonalization of $D^2u$, it is easy to show $J^{-1}\widehat{J}=\widehat{J}J^{-1}$. Since
\begin{equation}\aligned\label{yibxj}
\f{\p\bar{y}_i}{\p\bar{x}_j}=\sum_{k=1}^n\f{\p\bar{y}_i}{\p x_k}\f{\p x_k}{\p\bar{x}_j},
\endaligned
\end{equation}
then
$\left(\f{\p\bar{y}_i}{\p\bar{x}_j}\right)=\widehat{J}J^{-1}=J^{-1}\widehat{J}$, i.e., $\left(\f{\p\bar{y}_i}{\p\bar{x}_j}\right)$ is symmetric.
From \eqref{DEFJ} and the convex $u(x)+\f{\La}2|x|^2$, it's clear that the determinant of $J$ is positive, i.e., $\mathrm{det}J>0$.
With the injective $\bar{x}:\, \Om\rightarrow \bar{x}(\Om)$, we conclude that $\bar{x}:\, \Om\rightarrow \bar{x}(\Om)$ is a diffeomorphism. 
In particular, $\bar{x}(\Om)$ is simply connected since $\Om$ is convex. From Frobenius' theorem (see Lemma 7.2.11 in \cite{X} for instance), there is a function $\bar{u}$ on $\bar{x}(\Om)$ such that
\begin{equation}
D\bar{u}\big|_{\bar{x}(x)}=\bar{y}(x)=\f1{\sqrt{4\La^2+1}}(-x+2\La Du(x)).
\end{equation}
From \eqref{yibxj}, we have
\begin{equation}\aligned\label{RLubu}
D^2\bar{u}\big|_{\bar{x}(x)}=J^{-1}\widehat{J}=(2\La I+D^2u(x))^{-1}(-I+2\La D^2u(x)),
\endaligned
\end{equation}
which is equivalent to
\begin{equation}\aligned
D^2u(x)=(2\La I-D^2\bar{u})^{-1}(I+2\La D^2\bar{u})\big|_{\bar{x}(x)}.
\endaligned
\end{equation}
Note that both of $D\bar{u}$ and $D^2\bar{u}$ are independent of the choice of $\bar{u}$.
From \eqref{lowerLa}, for all $(\bar{x},D\bar{u}(\bar{x}))\in F_\La(M)$ we have
\begin{equation}\aligned\label{BOUNDD2u}
-\f{2\La^2+1}{\La}\le D^2\bar{u}(\bar{x})\le2\La.
\endaligned
\end{equation}
Using \eqref{BOUNDD2u}, we immediately have a volume estimate for special Lagrangian graph $M$ as follows.
\begin{proposition}\label{vol}
Suppose that $u$ is a smooth solution to the special Lagrangian equation on an open convex set $\Om\subset\R^n$ with \eqref{lowerLa}, and $M=\{(x,Du(x))\in\R^n\times\R^n|\ x\in\Om\}\subset\R^n\times\R^n$. Then for each $R>0$
\begin{equation}\aligned
\mathcal{H}^n(M\cap \mathbf{B}_R)\le \omega_n(4\La^2+5+\La^{-2})^{\f n2}R^n.
\endaligned
\end{equation}
\end{proposition}
\begin{proof}
Let $F_\La:\, (x,y)\rightarrow(\hat{x},\hat{y})$ be the isometric mapping defined before, and $\bar{u}$ be the function defined on $\bar{x}(\Om)$ as before.
In other words, $D\bar{u}$ is the graphic function of $F_\La(M)$.
Since $\Pi(F_\La(M))=\bar{x}(M)$,
then from \eqref{BOUNDD2u} we have
\begin{equation}\aligned
&\mathcal{H}^n(F_\La(M)\cap \mathbf{B}_R)=\int_{\Pi(F_\La(M)\cap \mathbf{B}_R)}\sqrt{\det(I+D^2\bar{u}D^2\bar{u})}\\
\le&\int_{\Pi(F_\La(M)\cap \mathbf{B}_R)}\left(1+\f{(2\La^2+1)^2}{\La^2}\right)^{\f n2}\le (4\La^2+5+\La^{-2})^{\f n2}\mathcal{H}^n(B_R).
\endaligned
\end{equation}
Hence
\begin{equation}\aligned
\mathcal{H}^n(M\cap \mathbf{B}_R)=\mathcal{H}^n(F_\La(M)\cap \mathbf{B}_R)\le \omega_n(4\La^2+5+\La^{-2})^{\f n2}R^n.
\endaligned
\end{equation}
This completes the proof.
\end{proof}

\section{
Mean value inequality on special Lagrangian graphs}

Let $u$ be a smooth solution to the special Lagrangian equation \eqref{SL} on an open set $\Om\subset\R^n$, and
$M=\mathrm{graph}_{Du}\triangleq\{(x,Du(x))\in\R^n\times\R^n|\ x\in\Om\}$ be a special Lagrangian graph over $\Om$ with $\mathbf{0}\in M$, $\p M\subset\p \mathbf{B}_R\subset\R^n\times\R^n$.
Let $\na$ be the Levi-Civita connection of $M$ with the induced metric $(\de_{ij}+\sum_{k=1}^nu_{ik}u_{jk})dx_idx_j$ from $\R^n\times\R^n$.
Let $\De_M$ denote the Laplacian of $M$ with this induced metric.
Here, $\p_{ik}u$ denotes the derivative of $u$ with respect to $x_i,x_k$.
Recall Sobolev inequality on minimal submanifolds proved by Michael-Simon \cite{MS} (see also \cite{B} by Brendle):
\begin{equation}\aligned\label{Sob}
\left(\int_{M}|\varphi|^{\f n{n-1}}\right)^{\f{n-1}n}\le c_{n}\int_M|\na\varphi|
\endaligned
\end{equation}
for any function $\varphi\in W^{1,1}_0(M)$, where $c_n\ge1$ is a constant depending only on $n$.
For any nonnegative subharmonic function on $M$, there holds
the mean value inequality (\cite{GT}\cite{MS}). Furthermore, if $\De_M\psi\ge-\th\psi$ for a nonnegative function $\psi$ on $M$ with a constant $\th\ge0$, then from Corollary 1.16 in \cite{CM1}, it follows that
\begin{equation}\aligned\label{MVIsub}
\psi(0)\le\f{e^{\f12\th r^2}}{\omega_n r^n}\int_{M\cap\mathbf{B}_{r}}\psi
\endaligned
\end{equation}
for any $0<r<R$.
Let $\k\ge1$ be a constant such that
\begin{equation}\aligned\label{DEFk}
\det(I+D^2uD^2u)\le \k^2\qquad \mathrm{on}\ \ \Om.
\endaligned
\end{equation}
Then $\p M\subset\p \mathbf{B}_R$ implies that $\Om$ contains a ball centered at the origin with the radius $R/\k$.
Since the Neumann-Poincar$\mathrm{\acute{e}}$ inequality holds on $\R^n$, then for any open set $V\subset B_r$ with rectifiable boundary $\p V$ and $r>0$, there holds
\begin{equation}\aligned
\min\left\{\mathcal{H}^{n}(V),\mathcal{H}^{n}(B_r\setminus V)\right\}\le c_n r\mathcal{H}^{n-1}(B_r\cap\p V)
\endaligned
\end{equation}
up to a choice of the constant $c_n\ge1$.
For any $\k r\le R$,
let $U$ be an open set in $M\cap\mathbf{B}_{\k r}$ with rectifiable boundary,
then $\Pi(U\cap\p \mathbf{B}_{\k r})\cap B_r=\emptyset$, and
\begin{equation}\aligned
\min\{\mathcal{H}^n(\Pi(U\cap\mathbf{B}_{r})),\mathcal{H}^n(\Pi(\mathbf{B}_{r}\setminus U))\}
\le c_n r\mathcal{H}^{n-1}(B_r\cap\p(\Pi(U))).
\endaligned
\end{equation}
Combining \eqref{DEFk} and $B_r\subset\Pi(\mathbf{B}_{\k r})$, we get
\begin{equation}\aligned
&\min\{\mathcal{H}^{n}(U\cap\mathbf{B}_{r}),\mathcal{H}^{n}(\mathbf{B}_{r}\setminus U)\}\\
\le& \min\left\{\int_{\Pi(U\cap\mathbf{B}_{r})}\sqrt{\det(I+D^2uD^2u)}dx,\int_{\Pi(\mathbf{B}_{r}\setminus U)}\sqrt{\det(I+D^2uD^2u)}dx\right\}\\
\le&\k\min\{\mathcal{H}^n(\Pi(U\cap\mathbf{B}_{r})),\mathcal{H}^n(\Pi(\mathbf{B}_{r}\setminus U))\}
\le c_n\k r\mathcal{H}^{n-1}(\mathbf{B}_{\k r}\cap\p U).
\endaligned
\end{equation}
By a standard argument (see Lemma 3.5 in \cite{D1} for instance), we have a Neumann-Poincar$\mathrm{\acute{e}}$ inequality on exterior balls as follows.
\begin{lemma}
\begin{equation}\aligned\label{NP}
\int_{M\cap\mathbf{B}_{r}}|f-\bar{f}_{r}|\le 2c_n\k r\int_{M\cap\mathbf{B}_{\k r}}|\na f|
\endaligned
\end{equation}
for all function $f\in W^{1,1}(M\cap\mathbf{B}_{\k r})$,
where $\bar{f}_{r}=\f1{\mathcal{H}^{n}(M\cap\mathbf{B}_{r})}\int_{M\cap\mathbf{B}_{r}}f$.
\end{lemma}

Using \eqref{Sob}\eqref{MVIsub}\eqref{NP}, we can get the mean value inequality for superharmonic functions on $M$ as follows.
\begin{theorem}\label{MVIsuper}
Let $M$ be the special Lagrangian graph defined previously in this section and $\k$ be the constant in \eqref{DEFk}.
Suppose that $\phi$ is a positive function satisfying $\De_M\phi\le \be\phi$ on $M$ for some constant $\be>0$. Then $\phi$ satisfies mean value inequality as follows:
\begin{equation}\aligned
\int_{M\cap\mathbf{B}_{\r}}\phi^{\de_n}\le c_{\k,\be\r^2}\mathcal{H}^{n}(M\cap\mathbf{B}_{\r})\phi^{\de_n}(\mathbf{0})
\endaligned
\end{equation}
for any $\r\in(0,R/2]$, where $\de_n\in(0,1]$ is a constant depending on $n$, and $c_{\k,\be\r^2}$ is a positive constant depending only on $n,\k,\be\r^2$.
\end{theorem}
The proof uses the famous De Giorgi-Nash-Moser iteration (refer \cite{D1}).
For self-containment and figuring out the constant $c_{\k,\be\r^2}$, we shall give the detailed proof of Theorem \ref{MVIsuper} here.  If the reader is quite familiar with it, one can skip the proof.

\begin{proof}
For any $r\in(0,\f{R}{2\k}]$, let
$$w=\log \phi-\f1{\mathcal{H}^{n}\left(M\cap\mathbf{B}_{\f32r}\right)}\int_{M\cap\mathbf{B}_{\f32r}}\log \phi,$$
then $\De_M\phi\le \be\phi$ implies
\begin{equation}\aligned\label{Dew}
\De_M w\le \be-|\na w|^2.
\endaligned
\end{equation}
Let $\e$ be a Lipschitz function with compact support in $M\cap \mathbf{B}_{r}$.
From \eqref{Dew}, for any $q\ge0$ integrating by parts implies
\begin{equation}\aligned
\int \left(|\na w|^2-\be\right)\e^2|w|^q\le&-\int\e^2|w|^q\De_M w=2\int\e|w|^q\na\e\cdot\na w+q\int\e^2|w|^{q-2}w|\na w|^2\\
\le&\f12\int |\na w|^2\e^2|w|^q+2\int|\na\e|^2|w|^q+q\int\e^2|w|^{q-1}|\na w|^2.
\endaligned
\end{equation}
Then
\begin{equation}\aligned\label{ewDwq}
\int \e^2|w|^q|\na w|^2\le2\int\left(2|\na\e|^2+\be\e^2\right)|w|^q+2q\int\e^2|w|^{q-1}|\na w|^2.
\endaligned
\end{equation}
We choose $\e_0=1$ on $\mathbf{B}_{\f32\k r}$, $\e_0=\f {4\k r-2|\mathbf{x}|}{\k r}$ on $\mathbf{B}_{2\k r}\setminus\mathbf{B}_{\f32\k r}$,
$\e_0=0$ outside $\mathbf{B}_{2\k r}$.
Then $|\na\e_0|\le \f2{\k r}$ on $M\cap\mathbf{B}_{2\k r}$.
Choosing $q=0$ in \eqref{ewDwq}, we have
\begin{equation}\aligned\label{|naw|2}
\int_{M\cap\mathbf{B}_{\f32\k r}}|\na w|^2\le2\int\left(2|\na\e_0|^2+\be\e_0^2\right)\le\left(16\k^{-2}r^{-2}+2\be\right)\mathcal{H}^{n}(M\cap\mathbf{B}_{2\k r}).
\endaligned
\end{equation}
Combining the Neumann-Poincar$\mathrm{\acute{e}}$ inequality \eqref{NP} for $w$, we have
\begin{equation}\aligned\label{|w|local}
&\int_{M\cap\mathbf{B}_{\f{3r}{2}}} |w|\le 2c_n\k\f{3r}{2}\int_{M\cap\mathbf{B}_{\f{3}2\k r}} |\na w|\\
\le&3c_n\k r\left(\mathcal{H}^n\left(M\cap\mathbf{B}_{\f{3}2\k r}\right)\right)^{\f12}\left(\int_{M\cap\mathbf{B}_{\f{3}2\k r}} |\na w|^2\right)^{\f12}\\
\le& 3c_n\k r\left(\mathcal{H}^n\left(M\cap\mathbf{B}_{2\k r}\right)\right)^{\f12}
\left(\left(16\k^{-2}r^{-2}+2\be\right)\mathcal{H}^{n}(M\cap\mathbf{B}_{2\k r})\right)^{\f12}\\
\le&12c_n(1+\be \k^2r^2)^{\f12}\mathcal{H}^n\left(M\cap\mathbf{B}_{2\k r}\right).
\endaligned
\end{equation}
With the definition of $\k$ in \eqref{DEFk}, we have
\begin{equation}\aligned\label{Br|w|}
\int_{M\cap\mathbf{B}_{\f32r}} |w|\le& c_{n}^*\omega_n(1+\be^{\f12} r)\k^{n+2}r^n
\endaligned
\end{equation}
for some constant $c_{n}^*$ depending only on $n$.

Denote $\bar{\be}=1+\be^{\f12}r$ for convenience.
Let $r_j=(1+2^{-j-1})r$ for each integer $j\ge0$.
Let $\e_j$ be the cut-off function on $\R^n\times\R^n$ such that $\e_j=1$ on $\mathbf{B}_{r_{j+1}}$, $\e_j=\f {r_{j}-|\mathbf{x}|}{r_{j}-r_{j+1}}$ on $\mathbf{B}_{r_{j}}\setminus \mathbf{B}_{r_{j+1}}$,
$\e_j=0$ outside $\mathbf{B}_{r_j}$.
Then $|\na\e_j|\le 2^{j+2}/r$.
From \eqref{ewDwq}, for any number $q\ge1$ and any integer $j\ge0$ we have
\begin{equation}\aligned\label{Brj-1e2wqnew2}
\int_{M\cap\mathbf{B}_{r_j}}\e_j^2|w|^q|\na w|^2\le 2^{2j+6}\f{\bar{\be}^2}{r^2}\int_{M\cap\mathbf{B}_{r_j}}|w|^q+2q\int_{M\cap\mathbf{B}_{r_j}}\e_j^2|w|^{q-1}|\na w|^2.
\endaligned
\end{equation}
Recall Young's inequality:
\begin{equation}\aligned\label{Young}
2q|w|^{q-1}\le\f12|w|^q+2^{2q-1}(q-1)^{q-1}\ \quad \mathrm{for}\ \ q\ge1,\\
\endaligned
\end{equation}
where we denote $0^0=1$ for the case $q=1$.
Combining \eqref{Brj-1e2wqnew2} with $q=0$, we get
\begin{equation}\aligned
\f12\int_{M\cap\mathbf{B}_{r_j}}\e_j^2|w|^q|\na w|^2\le& 2^{2j+6}\f{\bar{\be}^2}{r^2}\int_{M\cap\mathbf{B}_{r_j}}|w|^q+2^{2q-1}(q-1)^{q-1}\int_{M\cap\mathbf{B}_{r_j}}\e_j^2|\na w|^2\\
\le&2^{2j+6}\f{\bar{\be}^2}{r^2}\int_{M\cap\mathbf{B}_{r_j}}|w|^q+2^{2q+2j+5}q^{q-1}\f{\bar{\be}^2}{r^2}\mathcal{H}^{n}(M\cap\mathbf{B}_{r_j}).\\
\endaligned
\end{equation}
Combining Cauchy inequality and \eqref{Young}, for $q\ge1$ and $j\ge0$ we have
\begin{equation}\aligned\label{Brjwqnaw}
&\int_{M\cap\mathbf{B}_{r_{j}}}\e_j^2|w|^{q}|\na w|\le \f{r}{2^{j+5}\bar{\be}}\int_{M\cap\mathbf{B}_{r_{j}}}\e_j^2|w|^{q}|\na w|^2+\f{2^{j+3}\bar{\be}}{r}\int_{M\cap\mathbf{B}_{r_{j}}}\e_j^2|w|^{q}\\
\le&2^{j+2}\f{\bar{\be}}{r}\int_{M\cap\mathbf{B}_{r_{j}}}|w|^{q}+2^{2q+j+1}q^{q-1}\f{\bar{\be}}{r}\mathcal{H}^{n}(M\cap\mathbf{B}_{r_j})
+\f{2^{j+3}\bar{\be}}{r}\int_{M\cap\mathbf{B}_{r_{j}}}|w|^{q}\\
\le&2^{j+4}\f{\bar{\be}}{r}\int_{M\cap\mathbf{B}_{r_{j}}}\f{|w|^{q+1}+2^{2q+2}q^q}{4(q+1)}
+2^{2q+j+1}q^{q-1}\f{\bar{\be}}{r}\mathcal{H}^{n}(M\cap\mathbf{B}_{r_j})\\
\le&\f{2^{j+2}\bar{\be}}{(q+1)r}\int_{M\cap\mathbf{B}_{r_{j}}}|w|^{q+1}+2^{2q+j+5}q^{q-1}\f{\bar{\be}}{r}\mathcal{H}^{n}(M\cap\mathbf{B}_{r_j}).
\endaligned
\end{equation}
Moreover, for $j\ge0$ and $0\le q<1$, combining \eqref{Brj-1e2wqnew2} with $q=0$ and Young inequality
\begin{equation}\aligned\label{Brjwqnaw*}
&\int_{M\cap\mathbf{B}_{r_{j+1}}}|w|^q|\na w|\le \f{2^{j+1}\bar{\be}}{r}\int_{M\cap\mathbf{B}_{r_{j+1}}}|w|^{2q}+\f{r}{2^{j+3}\bar{\be}}\int_{M\cap\mathbf{B}_{r_{j}}}\e_j^2|\na w|^2\\
\le&\f{2^{j+1}\bar{\be}}{r}\int_{M\cap\mathbf{B}_{r_{j+1}}}\left(\f{2q}{q+1}|w|^{q+1}+\f{1-q}{q+1}\right)+2^{j+3}\f{\bar{\be}}{r}\mathcal{H}^{n}(M\cap\mathbf{B}_{r_j})\\
\le&\f{2^{j+2}\bar{\be}}{(q+1)r}\int_{M\cap\mathbf{B}_{r_{j}}}|w|^{q+1}+2^{j+4}\f{\bar{\be}}{r}\mathcal{H}^n\left(M\cap\mathbf{B}_{r_j}\right).
\endaligned
\end{equation}
Combining \eqref{Brjwqnaw} and \eqref{Brjwqnaw*}, we get
\begin{equation}\aligned\label{Brjwqnaw**}
\int_{M\cap\mathbf{B}_{r_{j+1}}}|w|^{q}|\na w|\le\f{2^{j+2}\bar{\be}}{(q+1)r}\int_{M\cap\mathbf{B}_{r_{j}}}|w|^{q+1}+2^{2q+j+5}q^{q-1}\f{\bar{\be}}{r}\mathcal{H}^{n}(M\cap\mathbf{B}_{r_j})
\endaligned
\end{equation}
for $q\ge0$ and $j\ge0$.
Combining Sobolev inequality \eqref{Sob} and \eqref{Brjwqnaw**}, for $j\ge0$ and $q\ge0$, we have
\begin{equation}\aligned
&\bigg(\int_{M\cap\mathbf{B}_{r_{j+2}}}|w|^{\f {(q+1)n}{n-1}}\bigg)^{\f{n-1}n}\le\left(\int_M\left(|w|^{q+1}\e_{j+1}\right)^{\f {n}{n-1}}\right)^{\f{n-1}n}\le c_n\int_M\left|\na(w^{q+1}\e_{j+1})\right|\\
&\le c_n \left((q+1)\int_{M\cap\mathbf{B}_{r_{j+1}}}|w|^{q}|\na w|+\f{2^{j+3}}{r}\int_{M\cap\mathbf{B}_{r_{j+1}}}|w|^{q+1}\right)\\
&\le \f{2^{j+2}c_n\bar{\be}}r\left(\int_{M\cap\mathbf{B}_{r_j}}|w|^{q+1}+2^{2q+3}(q+1)^{q}\mathcal{H}^n\left(M\cap\mathbf{B}_{r_j}\right)
+2\int_{M\cap\mathbf{B}_{r_{j+1}}}|w|^{q+1}\right)\\
&\le\f{2^{j+4}c_n\bar{\be}}r\left(\int_{M\cap\mathbf{B}_{r_j}}|w|^{q+1}+2^{2q+1}(q+1)^{q}\mathcal{H}^n\left(M\cap\mathbf{B}_{r_j}\right)\right).
\endaligned
\end{equation}
For any $f\in L^p(M\cap\mathbf{B}_{r_j})$ with $p>0$, $j\ge0$, we define $||f||_{p,r_j}=\left(\f1{\k\omega_nr_0^n}\int_{M\cap\mathbf{B}_{r_j}}|f|^p\right)^{\f1p}$.
Note that $\mathcal{H}^n\left(M\cap\mathbf{B}_{r_j}\right)\le\k\omega_nr_j^n\le\k\omega_nr_0^n$ by the definition of $\k$ in \eqref{DEFk}.
Then
\begin{equation}\aligned
||w||^q_{\f {nq}{n-1},r_{j+2}}\le 2^{j+5}c_n\bar{\be}(\k\omega_n)^{\f1n}\left(||w||^q_{q,r_{j}}+2^{2q}q^q\right)
\endaligned
\end{equation}
for any $j\ge0$, $q\ge1$, which implies
\begin{equation}\aligned\label{waqq-1r}
||w||_{\f{nq}{n-1} ,r_{j+2}}\le \left(2^{j+5}c_n\bar{\be}(\k\omega_n)^{\f1n}\right)^{\f1q}\left(||w||_{q,r_{j}}+4q\right).
\endaligned
\end{equation}
Let $q_j=\left(\f n{n-1}\right)^j$ and $a_j=||w||_{q_j,r_{2j}}/q_j$ for $j\ge0$. Then from \eqref{waqq-1r} we have
\begin{equation}\aligned
||w||_{q_{j+1} ,r_{2j+2}}\le (c_n\bar{\be}\omega_n^{1/n}\k^{1/n})^{\f1{q_j}}2^{\f {2j+5}{q_j}}\left(||w||_{q_j,r_{2j}}+4q_j\right),
\endaligned
\end{equation}
and
\begin{equation}\aligned\label{aj+1aj}
a_{j+1}\le \f{n-1}{n}(c_n\bar{\be}\omega_n^{1/n}\k^{1/n})^{\f1{q_j}}2^{\f {2j+5}{q_j}}(a_j+4)
\endaligned
\end{equation}
for every $j\ge0$.
Put $b_j=\f{n-1}{n}(c_n\bar{\be}\omega_n^{1/n}\k^{1/n})^{\f1{q_j}}2^{\f {2j+5}{q_j}}$. Then for each $j\ge0$ we have
\begin{equation}\aligned
a_{j+1}\le a_0\prod_{i=0}^jb_i+4\sum_{i=0}^j\prod_{k=i}^jb_k.
\endaligned
\end{equation}
There is a positive constant $b_*$ depending only on $n$ such that for all $j\ge i\ge0$
\begin{equation}\aligned
\prod_{k=i}^jb_k\le b_*\left(\f{n-1}{n}\right)^{j-i+1}\bar{\be}^{\f n{q_i}}\k^{\f1{q_i}}.
\endaligned
\end{equation}
Hence for each $j\ge1$
\begin{equation}\aligned
a_{j+1}\le \left(\f{n-1}{n}\right)^{j+1}b_*\bar{\be}^n\k a_0+4b_*\sum_{i=0}^{j}\left(\f{n-1}{n}\right)^{j-i+1}\bar{\be}^{\f n{q_i}}\k^{\f1{q_i}}\\
\endaligned
\end{equation}
and then
\begin{equation}\aligned
&||w||_{q_j,r_{2j}}\le b_*\bar{\be}^n\k a_0+4b_*\sum_{i=0}^{j-1}\left(\f{n}{n-1}\right)^{i}\bar{\be}^{\f n{q_i}}\k^{\f1{q_i}}.
\endaligned
\end{equation}
Denote $\bar{\k}=\bar{\be}^n\k$, and $i_*=\left[\f{\log(1+\log\bar{\k})}{\log n/(n-1)}\right]$. Then
\begin{equation}\aligned
||w||_{q_j,r_{2j}}\le &b_*\bar{\k} a_0+4b_*\sum_{i=0}^{i_*}\left(\f{n}{n-1}\right)^{i}\bar{\k}+4b_*\sum_{i=i_*+1}^{j-1}\left(\f{n}{n-1}\right)^{i}\bar{\k}^{\left(\f{n-1}n\right)^{\f{\log(1+\log\bar{\k})}{\log n/(n-1)}}}\\
\le& b_*\bar{\k} a_0+4(n-1)b_*\left(\f{n}{n-1}\right)^{i_*+1}\bar{\k}+4(n-1)b_*\left(\f{n}{n-1}\right)^{j}\bar{\k}^{\f{1}{1+\log \bar{\k}}}\\
\le& b_*\bar{\k} a_0+4nb_*(1+\log\bar{\k})\bar{\k}+4(n-1)b_*q_je^{\f{\log\bar{\k}}{1+\log \bar{\k}}}.
\endaligned
\end{equation}
For each integer $k\ge 1$, there is an integer $j_k\ge0$ such that $q_{j_k}\le k\le q_{j_k+1}$. Note that $\mathcal{H}^n(M\cap \mathbf{B}_{r_j})\le\k\omega_nr_0^n$ by the definition of $\k$ in \eqref{DEFk}. With H$\mathrm{\ddot{o}}$lder inequality, we have
\begin{equation}\aligned
||w||_{k,r}\le||w||_{k,r_{2j_k+2}}&\le ||w||_{q_{j_k+1},r_{2j_k+2}}\le b_*\bar{\k} a_0+4nb_*(1+\log\bar{\k})\bar{\k}+4e(n-1)b_*q_{j_k+1}\\
\le& b_*\bar{\k} a_0+4nb_*(1+\log\bar{\k})\bar{\k}+4enb_*k.
\endaligned
\end{equation}
Note that
$$a_0=||w||_{1,r_0}\le\f{c_n^*\bar{\be}\k^{n+2}\omega_nr^n}{\k\omega_nr_0^n}\le c_n^*\bar{\be}\k^{n+1}$$
from \eqref{Br|w|}. Then there is a constant $\de_n\in(0,1]$ depending only on $n$ such that
\begin{equation}\aligned
||w||_{k,r}\le \f1{2\de_n }\left(\bar{\k}^2\k^{n}+\f{k}{2e}\right)
\endaligned
\end{equation}
for all integers $k\ge1$.
Therefore, combining Stirling's formula
\begin{equation}\aligned
||w||_{k,r}^{k}\le \de_n^{-k}\left(\bar{\k}^{2k}\k^{nk}+(2e)^{-k}k^{k}\right)\le \de_n^{-k}\left(\bar{\k}^{2k}\k^{nk}+2^{-k}k^{-\f12}k!\right),
\endaligned
\end{equation}
which implies
\begin{equation}\aligned
\sum_{k=1}^\infty\f{1}{k!}\left|\left|\de_nw\right|\right|_{k,r}^{k}\le\sum_{k=1}^\infty \f{\bar{\k}^{2k}\k^{nk}}{k!}+\sum_{k=1}^\infty 2^{-k}k^{-\f12}\le e^{\bar{\k}^2\k^{n}}.
\endaligned
\end{equation}
Hence
\begin{equation}\aligned
\f1{(\bar{\k}\omega_nr_0^n)^2}\int_{M\cap\mathbf{B}_{r}} e^{\de_nw}\int_{M\cap\mathbf{B}_{r}} e^{-\de_nw}\le \left(e^{\bar{\k}^2\k^{n}}+1\right)^2.
\endaligned
\end{equation}
Namely,
\begin{equation}\aligned\label{phic*-c*}
\int_{M\cap\mathbf{B}_{r}} \phi^{\de_n}\int_{M\cap\mathbf{B}_{r}} \phi^{-\de_n}\le (\bar{\k}\omega_nr_0^n)^2\left(e^{\bar{\k}^2\k^{n}}+1\right)^2\le 2(\bar{\k}\omega_nr_0^n)^2e^{2\bar{\k}^2\k^{n}}.
\endaligned
\end{equation}
From $\De_M\phi\le \be\phi$ and $\phi>0$, we have 
$$\De_M\phi^{-\de_n}=-\de_n\phi^{-1-\de_n}\De_M\phi+\de_n(1+\de_n)\phi^{-2-\de_n}|\na\phi|^2\ge-\de_n\be\phi^{-\de_n}.$$
Then with \eqref{MVIsub} we have
\begin{equation}\aligned\label{Mphic*}
\phi^{-\de_n}(\mathbf{0})\le\f{e^{\f12\de_n\be r^2}}{\omega_n r^n}\int_{M\cap\mathbf{B}_{r}}\phi^{-\de_n}.
\endaligned
\end{equation}
Note $r_0=\f32r$ and $r\in(0,\f{R}{2\k}]$. Then combining \eqref{phic*-c*}\eqref{Mphic*} gets
\begin{equation}\aligned\label{ula14r}
\int_{M\cap\mathbf{B}_r} \phi^{\de_n}\le 2\left(\f32\right)^{2n}\omega_n\bar{\k}^2e^{2\bar{\k}^2\k^{n}}e^{\f12\de_n\be r^2}r^n\phi^{\de_n}(\mathbf{0})\quad \mathrm{for\ any}\ r\in(0,\f{R}{2\k}\big],
\endaligned
\end{equation}
where $\bar{\k}=\left(1+\be^{\f12}r\right)^n\k$.

Now we fix a constant $\r\in(0,R/2]$ and $r=\f{\r}{2\k}$. From \eqref{ula14r}, we conclude that there is a constant $\th_{\k,\be\r^2}\ge1$ depending only on $n,\k,\be\r^2$ so that
\begin{equation}\aligned\label{MBrphideth}
\int_{M\cap\mathbf{B}_{r}(\mathbf{x})} \phi^{\de_n}\le \th_{\k,\be \r^2}r^n\phi^{\de_n}(\mathbf{x})
\endaligned
\end{equation}
for any $\mathbf{x}\in M\cap\overline{\mathbf{B}_\r}$. Hence, there is a constant $\th'_{\k,\be\r^2}\ge\th_{\k,\be\r^2}$ depending only on $n,\k,\be\r^2$ such that 
\begin{equation}\aligned\label{MBrphideth'}
\inf_{M\cap\mathbf{B}_{r/4}(\mathbf{y})} \phi^{\de_n}\le \th'_{\k,\be \r^2}\phi^{\de_n}(\mathbf{x})
\endaligned
\end{equation}
for any $\mathbf{x}\in M\cap\overline{\mathbf{B}_\r}$ and any $\mathbf{y}\in M\cap\overline{\mathbf{B}_{3r/4}(\mathbf{x})}$. 
Denote $\mathbf{y}_0=0$.
By induction, there are an integer $n_\k\ge1$ depending only on $n,\k$ and a collection of points $\mathbf{y}_1,\cdots,\mathbf{y}_{n_\k}\in M\cap\overline{\mathbf{B}_{\r-r/4}}$ with $\inf_{0\le j\le i-1}|\mathbf{y}_i-\mathbf{y}_j|<\f r2$ for each $i=1,\cdots,n_\k$ such that
\begin{equation}\aligned
M\cap\overline{\mathbf{B}_\r}\subset\bigcup_{j=1}^{n_\k}\mathbf{B}_{\f34r}(\mathbf{y}_j).
\endaligned
\end{equation}
Denote $\mathbf{z}_0=0$.
By induction, from \eqref{MBrphideth'} we choose a sequence of points $\mathbf{z}_i\in\overline{\mathbf{B}_{r/4}(\mathbf{y}_i)}\subset M\cap\overline{\mathbf{B}_{\r}}$ such that
\begin{equation}\aligned\label{MBrphideth'z}
\phi^{\de_n}(\mathbf{z}_i)\le \inf_{0\le j\le i-1}\th'_{\k,\be \r^2}\phi^{\de_n}(\mathbf{z}_j)
\endaligned
\end{equation}
for each $i=1,\cdots,n_\k$.
So we get
$M\cap\overline{\mathbf{B}_\r}\subset\cup_{j=1}^{n_\k}\mathbf{B}_{r}(\mathbf{z}_j)$ and $\phi^{\de_n}(\mathbf{z}_i)\le (\th'_{\k,\be \r^2})^i\phi^{\de_n}(\mathbf{0})$ for each $i=1,\cdots,n_\k$. Hence, with \eqref{MBrphideth} we deduce
\begin{equation}\aligned
\int_{M\cap\mathbf{B}_{\r}} \phi^{\de_n}\le\sum_{j=1}^{m_\k}\int_{M\cap\mathbf{B}_{r}(\mathbf{z}_j)}\phi^{\de_n}\le\th_{\k,\be \r^2}r^n\sum_{j=1}^{m_\k}\phi^{\de_n}(\mathbf{z}_j)\le\th_{\k,\be \r^2}r^n\sum_{j=1}^{m_\k}(\th'_{\k,\be \r^2})^j\phi^{\de_n}(\mathbf{0}).
\endaligned
\end{equation}
Namely, there is a constant $c_{\k,\be \r^2}$ depending only on $n,\k,\be\r^2$ such that
\begin{equation}\aligned
\int_{M\cap\mathbf{B}_{\r}} \phi^{\de_n}\le c_{\k,\be \r^2}\r^n\phi^{\de_n}(\mathbf{0}).
\endaligned
\end{equation}
This completes the proof of Theorem \ref{MVIsuper}.
\end{proof}

\section{A Liouville type theorem for special Lagrangian equations}

Let $u$ be a smooth solution to the special Lagrangian equation \eqref{SL} with the phase $\Th$ on $B_R\subset\R^n$.
Without loss of generality, we assume the constant $\Th\ge0$.
Let $M=\{(x,Du(x))\in\R^n\times\R^n|\ x\in B_R\}$. Denote $g_{ij}=\de_{ij}+\sum_ku_{ik}u_{jk}$, and $v=\sqrt{\det g_{ij}}$.
We usually see $v$ as a function on $M$ by identifying $v(x,Du(x))=v(x)$, which
will not cause confusion from the context in general.
Let $\la_1,\cdots,\la_n$ be the eigenvalues of the Hessian $D^2u$ on $B_R$.
Let $\De_M$ denote the Laplacian of $M$, and $\na_{M}$ denote the Levi-Civita connection of $M$.
Let $\p_{ij}u$ denote the derivative of $u$ with respect to $x_i,x_j$, and $\p_{ijk}u$ denote the derivative of $u$ with respect to $x_i,x_j,x_k$.
At any considered point $p$, we assume that $D^2u$ is diagonalized, then
\begin{equation}\aligned\label{Delogv}
\De_M\log v=\sum_{i,j,k}(1+\la_i\la_j)h_{ijk}^2
\endaligned
\end{equation}
at $p$, where $h_{ijk}=\f1{\sqrt{(1+\la_i^2)(1+\la_j^2)(1+\la_k^2)}}\p_{ijk}u$ (see \cite{WmY2} for instance).
Let $\bn$ be Levi-Civita connection of $\R^n\times\R^n$ with respect to its standard metric.
Let $E_1,\cdots,E_{2n}$ be the orthonormal basis of $\R^n\times\R^n$ such that $E_i$ is the dual form of $dx_i$, and $E_{n+i}$ is the dual form of $dy_i$ for each $i=1,\cdots,n$.
Let $e_1,\cdots,e_n$ be a local tangent frame in a neighborhood of $p$ defined by
$$e_i=\f1{\sqrt{1+|Du_i|^2}}(E_i+\p_{ik} u E_{n+k}),$$
and $\nu_1,\cdots,\nu_n$ be a local frame normal to $M$ in a neighborhood of $p$ defined by
$$\nu_j=\f1{\sqrt{1+|Du_j|^2}}(-\p_{jk}uE_k+E_{n+j}).$$
Then at the point $p$,
$$e_i=\f1{\sqrt{1+\la_i^2}}(E_i+\la_iE_{n+i}),\qquad \nu_j=\f1{\sqrt{1+\la_j^2}}(-\la_j E_j+E_{n+j}),$$
and they make up an orthonormal basis of $\R^n\times\R^n$.
Let $B_M$ denote the second fundamental form on $M$, 
then at $p$ we have
\begin{equation}\aligned
\lan B_M(e_i,e_j),\nu_k\ran=\lan \bn_{e_i}e_j,\nu_k\ran=\f1{\sqrt{(1+\la_i^2)(1+\la_j^2)(1+\la_k^2)}}\p_{ijk}u=h_{ijk}.
\endaligned
\end{equation}
Let $|B_M|^2$ denote the square norm of $B_M$, i.e.,
\begin{equation}\aligned
|B_M|^2=\sum_{i,j=1}^n|B_M(e_i,e_j)|^2=\sum_{i,j,k=1}^n|\lan B_M(e_i,e_j),\nu_k\ran|^2=\sum_{i,j,k=1}^nh_{ijk}^2.
\endaligned
\end{equation}
From \eqref{Delogv}, we have
\begin{equation}\aligned
\De_{M} v^{-\f1n}=&\De_{M} e^{-\f1n\log v}=-\f1n v^{-\f1n}\De_{M}\log v+\f1{n^2}v^{-\f1n}|\na_M\log v|^2\\
=&-\f1n v^{-\f1n}\left(\sum_{i,j,k}h^2_{ijk}+\sum_{k,i\neq j}\la_i\la_jh^2_{ijk}+\sum_{i,k}\la_i^2h^2_{iik}-\f1n\sum_{i,j,k}\la_i\la_jh_{iik}h_{jjk}\right).
\endaligned
\end{equation}
Suppose there is a constant $K\ge1$ such that $\la_i\la_j+K\ge0$ for all $i,j$.
Combining Cauchy inequality, we have
\begin{equation}\aligned\label{v-1nA}
\De_{M} v^{-\f1n}\le&-\f1n v^{-\f1n}\left(\sum_{i,j,k}h^2_{ijk}-\sum_{k,i\neq j}Kh^2_{ijk}\right)\le \f{K-1}n v^{-\f1n}\sum_{i,j,k}h^2_{ijk}=\f{K-1}n v^{-\f1n}|B_M|^2.
\endaligned
\end{equation}

\begin{lemma}\label{mulowerbd}
Assume that $\mu_1,\cdots,\mu_n$ are constants with $\sum_i\arctan\mu_i\ge0$. If $K$ is a positive constant $\ge1$ such that $\mu_i\mu_j+K\ge0$ for all $i,j$, then $\mu_i\ge-\La_K$, where $\La_K$ is the unique solution to
\begin{equation}\aligned\label{arctant}
\arctan t=(n-1)\arctan\f Kt\qquad on \ \ (0,\infty).
\endaligned
\end{equation}
In particular, $\La_K< 2nK/\pi$.
\end{lemma}
\begin{proof}
Without loss of generality, we assume $\mu_1\ge\cdots\ge\mu_n$ and $\mu_n<0$, or else we have complete the proof.
From $\mu_1\mu_n+K\ge0$, we get $\mu_1\le-K/\mu_n$. Then
\begin{equation}\aligned
0\le&\sum_i\arctan\mu_i\le(n-1)\arctan\mu_1+\arctan\mu_n\\
\le&(n-1)\arctan\left(-\f K{\mu_n}\right)+\arctan\mu_n=\arctan\mu_n-(n-1)\arctan\left(\f K{\mu_n}\right).
\endaligned
\end{equation}
Since $\arctan t-(n-1)\arctan\left(\f Kt\right)$ is monotonic increasing on $(0,\infty)$, then the above inequality implies $\mu_i\ge\La_K$ for all $i=1,\cdots,n$, where $\La_K$ is the unique solution to \eqref{arctant}.

Now let us estimate the upper bound of $\La_K$. Since $\tan\left(\f{\pi}{2n}\right)>\f{\pi}{2n}$, then
\begin{equation}\aligned
\f{\pi}2-n\arctan\left(\f{\pi}{2n}\right)>0.
\endaligned
\end{equation}
Hence for $K\ge1$ we have
\begin{equation}\aligned
\arctan\left(\f{2nK}{\pi}\right)-(n-1)\arctan\left(\f{\pi}{2n}\right)=&\f{\pi}2-\arctan\left(\f{\pi}{2nK}\right)-(n-1)\arctan\left(\f{\pi}{2n}\right)\\
>&\f{\pi}2-n\arctan\left(\f{\pi}{2n}\right)>0,
\endaligned
\end{equation}
which implies $\La_K< 2nK/\pi$.
\end{proof}

Let
$$M_\La=F_\La(M)=\{F_\La(x,y)\in\R^n\times\R^n|\, (x,y)\in M\},$$
which is a rigid motion of $M$ with $F_\La$ defined in \eqref{labxby}.
Let $\De_{M_\La}$ denote the Laplacian of $M_\La$, and $\na_{M_\La}$ denote the Levi-Civita connection of $M_\La$.
\begin{lemma}\label{ISOtrans}
Let $\phi$ be a smooth function on $M_\La$, then $\phi_\La\triangleq\phi\circ F_\La$ satisfies
\begin{equation}\aligned
\De_M\phi_\La(\mathbf{x})=\De_{M_\La}\phi(F_\La(\mathbf{x})),\qquad |\na_{M}\phi_\La|(\mathbf{x})=|\na_{M_\La}\phi|(F_\La(\mathbf{x}))
\endaligned
\end{equation}
for any $\mathbf{x}\in M$.
\end{lemma}
\begin{proof}
Let $(z_1,\cdots,z_n)$ be a local coordinate chart in a neighborhood of the considered point $\mathbf{x}$ in $M$, such that $\p_{z_i}$ forms an orthonormal basis at $\mathbf{x}$.
Let $(w_1,\cdots,w_n)=F_\La(z_1,\cdots,z_n)$, then $(w_1,\cdots,w_n)$ is a local coordinate chart in a neighborhood of $F_\La(\mathbf{x})$ and
$\p_{w_i}=\sum_j\f{\p w_i}{\p z_j}\p_{z_j}$.
Since $F_\La$ is an isometric mapping, then at $F_\La(\mathbf{x})$
\begin{equation}\aligned
\lan\p_{w_i},\p_{w_j}\ran=\left\lan\sum_k\f{\p w_i}{\p z_k}\p_{z_k},\sum_l\f{\p w_j}{\p z_l}\p_{z_l}\right\ran=\sum_k\f{\p w_i}{\p z_k}\f{\p w_j}{\p z_k}=\de_{ij}.
\endaligned
\end{equation}
Recall that $\bn$ denotes Levi-Civita connection of $\R^n\times\R^n$ with respect to its standard metric.
For any function $\varphi\in C^\infty_c(M_\La)$,
\begin{equation}\aligned
&\left\lan\na_{M}\phi_\La,\na_{M}(\varphi\circ F_\La)\right\ran\big|_{\mathbf{x}}=\sum_i\bn_{z_i}\phi_\La\big|_{\mathbf{x}}\bn_{z_i}(\varphi\circ F_\La)\big|_{\mathbf{x}}\\
=&\sum_i\sum_j\f{\p w_j}{\p z_i}\bn_{w_j}\phi\big|_{F_\La(\mathbf{x})}\sum_k\f{\p w_k}{\p z_i}\bn_{w_k}\varphi\big|_{F_\La(\mathbf{x})}\\
=&\sum_j\bn_{w_j}\phi\big|_{F_\La(\mathbf{x})}\bn_{w_j}\varphi\big|_{F_\La(\mathbf{x})}
=\left\lan\na_{M_\La}\phi,\na_{M_\La}\varphi\right\ran\big|_{F_\La(\mathbf{x})}.
\endaligned
\end{equation}
Hence we have proved $|\na_{M}\phi_\La|(\mathbf{x})=|\na_{M_\La}\phi|(F_\La(\mathbf{x}))$.
Let $d\mu$ and $d\mu_\La$ denote the volume elements of $M$ and $M_\La$, respectively.
Integrating by parts infers
\begin{equation}\aligned
\int_{M_\La}\varphi \De_{M_\La}\phi\, d\mu_\La=&-\int_{M_\La}\left\lan\na_{M_\La}\phi,\na_{M_\La}\varphi\right\ran d\mu_\La\\
=&-\int_M\left\lan\na_{M}\phi_\La,\na_{M}(\varphi\circ F_\La)\right\ran d\mu=\int_M\varphi\circ F_\La\, \De_{M}\phi_\La\, d\mu,
\endaligned
\end{equation}
which implies $\De_M\phi_\La(\mathbf{x})=\De_{M_\La}\phi(F_\La(\mathbf{x}))$.
We complete the proof.
\end{proof}

Now let us introduce briefly the several notions from geometric measure theory (see \cite{LYa}\cite{S} for more details), which will be used in the following text.
For a set $S$ in Euclidean space $\R^{n+m}$, we call $S$ \emph{countably $n$-rectifiable} if $S\subset S_0\cup\bigcup_{k=1}^\infty F_k(\R^n)$, where $\mathcal{H}^n(S_0)=0$, and $F_k:\, \R^n\rightarrow\R^{n+m}$ are Lipschitz mappings for all integers $k\ge1$.
For an open set $U\subset\R^{n+m}$, a varifold $V$ in $U$ is a Radon measure on
$$G_n(U)=\{(x,T)|\, x\in U,\, T\ \mathrm{is\ an}\ n\mathrm{-dimensional\ subspace\ of}\ \R^{n+m}\}.$$
Associated to $V$, there is a Radon measure $\mu_V$ on $U$ defined by $V\circ\pi^{-1}$ with the projection $\pi:\,G_n(U)\rightarrow U$.
An $n$-rectifiable varifold in $U$ is a varifold in $U$ which is supported on countably $n$-rectifiable sets.
The multiplicity functions of varifolds can be defined through tangent spaces in the sense of Radon measures (see Definition 38.1 in \cite{S} for instance).
If a varifold $V$ has an integer-valued multiplicity function, we say that $V$ has \emph{integer multiplicity}, which can imply that spt$V$ is $n$-rectifiable (see Theorem 38.3 in \cite{S}).

Let $S$ be a countably $n$-rectifiable set in $\R^{n+m}$ with finite $n$-dimensional Hausdorff measure on any compact set of $\R^{n+m}$.
We use $|S|$ to denote the multiplicity one varifold associated with $S$, i.e., the \emph{$n$-rectifiable varifold} with the support $S$ and the multiplicity function one on $S$.
An $n$-varifold $V$ is said to be a \emph{stationary ($n$-)varifold} in open $U\subset\R^{n+m}$ if $V$ is an $n$-varifold in $U$ with
\begin{equation}\aligned\nonumber
\int \mathrm{div}_V Yd\mu_V=0
\endaligned
\end{equation}
for each $Y\in C^\infty_c(U,\R^{n+m})$. Here, $\mathrm{div}_V Y$ is the divergence of $Y$ restricted on spt$V$.
When we say an $n$-dimensional \emph{minimal cone} $C$ in $\R^{n+m}$, we mean that $C$ is an integer multiplicity stationary varifold with support being a cone.
\begin{lemma}\label{Multi1}
Let $u_k$ be a sequence of smooth solutions to \eqref{SL} on $\R^n$ with $-K\le D^2u_k\le K$ on $\R^n$ for some constant $K>0$. Denote $M_k=\{(x,Du_k(x))\in\R^n\times\R^n|\, x\in \R^n\}$. Suppose $(0,Du_k(0))\in M_k$.
Then there are a $C^{1,1}$-function $u_\infty$ on $\R^n$ with $-K\le D^2u_\infty\le K$ a.e. on $\R^n$ and a multiplicity one stationary varifold $V$ such that
up to a choice of the subsequence $|M_k|$ converges to $V$ in the varifold sense with $\mathrm{spt}V=\{(x,Du_\infty(x))\in\R^n\times\R^n|\, x\in \R^n\}$.
\end{lemma}
\begin{proof}
By Arzela-Ascoli theorem, up to a choice of the subsequence, we assume that there is a $C^{1,1}$-function $u_\infty$ on $\R^n$ with $-K\le D^2u_\infty\le K$ a.e. on $\R^n$ such that $u_k\rightarrow u_\infty$ uniformly on compact sets of $\R^n$ in $C^{1,\a}$-norm for any $\a\in(0,1)$. By compactness of varifolds (see Theorem 42.7 and Remark 42.8 in \cite{S}), up to a choice of the subsequence, we can assume $|M_k|$ converges to an integer multiplicity stationary varifold $V$ in $\R^n\times\R^n$ in the varifold sense. Let $\mu_V$ denote the Radon measure associated to $V$. By monotonicity of the density of $V$ (see formula 17.3 in \cite{S} for instance), we have
$\mu_V(\mathbf{B}_r(\mathbf{x}_*))\ge\omega_nr^n$ for any $\mathbf{x}_*\in\mathrm{spt}V$ and $r>0$. By varifold convergence of $|M_k|$, there is a sequence $\mathbf{x}_k\in M_k$ with $\mathbf{x}_k\rightarrow\mathbf{x}_*$. Denote $\mathbf{x}_k=(x_k,Du_k(x_k))$. Then $x_k$ converges to a point $x_*$ with $\Pi(\mathbf{x}_*)=x_*$, where $\Pi$ denotes the projection from $\R^n\times\R^n$ into $\R^n$ defined by $\Pi(\mathbf{x})=x$ for any $\mathbf{x}=(x,y)\in\R^n\times\R^n$ as before. Therefore, $\mathbf{x}_*=\lim_{k\rightarrow\infty}\mathbf{x}_k=\lim_{k\rightarrow\infty}(x_k,Du_k(x_k))=(x_*,Du_\infty(x_*))$, which implies the support of $V$
\begin{equation}\aligned\label{sptVVVVV}
\mathrm{spt}V\subset\{(x,Du_\infty(x))\in\R^n\times\R^n|\, x\in \R^n\}.
\endaligned
\end{equation}
Note that for any $z\in\R^n$, $\mathcal{H}^n(\mathbf{B}_r(\mathbf{z}_k)\cap M_k)\ge\omega_nr^n$ with $\mathbf{z}_k=(z,Du_k(z))$.
Since $(z,Du_k(z))\rightarrow(z,Du_\infty(z))$ as $k\rightarrow\infty$, from varifold convergence of $|M_k|$ we get $\mu_V(\mathbf{B}_r(\mathbf{z}))\ge\omega_nr^n$ for $\mathbf{z}=(z,Du_\infty(z))$.
In particular, $\mathbf{z}\in\mathrm{spt}V$, which implies
\begin{equation}\aligned\label{sptVVVVV*}
\{(x,Du_\infty(x))\in\R^n\times\R^n|\, x\in \R^n\}\subset\mathrm{spt}V.
\endaligned
\end{equation}

Now it only remains to prove that $V$ has multiplicity one. Let reg$V$ denote the regular part of $V$.
For any $\mathbf{y}\in\mathrm{reg}V$, let $T_{\mathbf{y}}V$ denote the tangent plane of spt$V$ at $\mathbf{y}$.
Let $\xi_1,\cdots,\xi_n$ be an orthonormal basis of $T_{\mathbf{y}}V$. 
From Lemma 22.2 in \cite{S} and Proposition \ref{vol},
\begin{equation}\aligned\label{Mkeix}
\lim_{r\rightarrow0}\left(r^{-n}\lim_{k\rightarrow\infty}\int_{M_k\cap \mathbf{B}_{r}(\mathbf{y})}\left|e_{k,1}\wedge\cdots\wedge e_{k,n}-\xi_{1}\wedge\cdots\wedge \xi_{n}\right|^2\right)=0,
\endaligned
\end{equation}
where $e_{k,1},\cdots,e_{k,n}$ is a local orthonormal tangent frame of $M_k$ for each $k$.
We also treat $e_{k,i}$ as a vector on $\Pi(M_k)$ by letting $e_{k,i}(x)=e_{k,i}(x,Du_k(x))$ for each $i=1,\cdots,n$ and $k\ge1$.
Let $v_k=\sqrt{\det(\de_{ij}+\sum_ku_{ik}u_{jk})}$, then $v_k^{-1}=\left|\lan e_{k,1}\wedge\cdots\wedge e_{k,n},E_1\wedge\cdots\wedge E_n\ran\right|$ with $E_1,\cdots,E_n$ being a standard orthonormal basis of $\R^n$. From $-K\le D^2u_k\le K$ on $\R^n$ and \eqref{Mkeix}, we get
\begin{equation}\aligned\label{Mkeix*}
\lim_{r\rightarrow0}\left(r^{-n}\lim_{k\rightarrow\infty}\int_{B_r(\Pi(\mathbf{y}))}\left|e_{k,1}\wedge\cdots\wedge e_{k,n}-\xi_{1}\wedge\cdots\wedge \xi_{n}\right|^2v_k\right)=0.
\endaligned
\end{equation}
Let $v_\infty$ be a positive constant defined by $\left|\lan \xi_1\wedge\cdots\wedge \xi_n,E_1\wedge\cdots\wedge E_n\ran\right|^{-1}$.
From \eqref{Mkeix*} and
\begin{equation}\aligned
&\int_{B_{r}(\Pi(\mathbf{y}))}\left|1-v_kv_\infty^{-1}\right|=\int_{B_{r}(\Pi(\mathbf{y}))}\left|v_k^{-1}-v_\infty^{-1}\right|v_k\\
\le&\int_{B_{r}(\Pi(\mathbf{y}))}\left|\lan e_{k,1}\wedge\cdots\wedge e_{k,n}-\xi_1\wedge\cdots\wedge \xi_n,E_1\wedge\cdots\wedge E_n\ran\right|v_k\\
\le&\int_{B_{r}(\Pi(\mathbf{y}))}\left|e_{k,1}\wedge\cdots\wedge e_{k,n}-\xi_1\wedge\cdots\wedge \xi_n\right|v_k,
\endaligned
\end{equation}
with Cauchy inequality we get
\begin{equation}\aligned
\lim_{r\rightarrow0}\left(r^{-n}\lim_{k\rightarrow\infty}\int_{B_{r}(\Pi(\mathbf{y}))}\left|1-v_kv_\infty^{-1}\right|\right)=0,
\endaligned
\end{equation}
which implies
\begin{equation}\aligned
&\lim_{r\rightarrow0}\left(r^{-n}\mu_V(B_{r}(\Pi(\mathbf{y}))\times\R^n)\right)
=\lim_{r\rightarrow0}\left(r^{-n}\lim_{k\rightarrow\infty}\mathcal{H}^n\left(M_k\cap(B_{r}(\Pi(\mathbf{y}))\times\R^n)\right)\right)\\
=&\lim_{r\rightarrow0}\left(r^{-n}\lim_{k\rightarrow\infty}\int_{B_{r}(\Pi(\mathbf{y}))}v_k\right)=\omega_n v_\infty=\omega_n\left|\lan \xi_1\wedge\cdots\wedge \xi_n,E_1\wedge\cdots\wedge E_n\ran\right|^{-1}.
\endaligned
\end{equation}
With \eqref{sptVVVVV}, we conclude that $V$ has multiplicity one everywhere on spt$V$. This completes the proof.
\end{proof}
We need a dimensional estimate for singular sets of non-smooth special Lagrangian graphs.
\begin{lemma}\label{SingHD}
Let $u$ be a $C^{1,1}$-function on $B_R\subset\R^n$ with $-K\le D^2u\le K$ a.e. on $B_R$ for some constant $K>0$, and $M=\{(x,Du(x))\in\R^n\times\R^n|\, x\in B_R\}$. If $|M|$ is stationary in $B_R\times\R^n$, then the singular set of $M$ is a closed set of Hausdorff dimension $\le n-4$ in $M$.
\end{lemma}
\textbf{Remark}. Here, a point $\mathbf{x}$ in the singular set of $M$ means that any tangent of $M$ at $\mathbf{x}$ is not an $n$-plane. Such a point is said to be a singular point.
If we write $\mathbf{x}=(x,Du(x))$ for the function $u$ in this lemma, then the singular point $\mathbf{x}$ of $M$ is equivalent to that $u$ is not $C^2$ at $x$.
\begin{proof}
The proof is the combination of Bernstein theorem for 3-dimensional minimal graphs (see Theorem 5.4 in \cite{Fc} or Theorem 1.3 in \cite{Y1}) and Federer's dimension reduction argument.
Let $\mathcal{S}$ denote the singular set of $M$. From Allard's regularity theorem \cite{A} (see also Theorem 24.2 in \cite{S}), $\mathcal{S}$ is a closed set in $M$.
We suppose that $\mathcal{S}$ has Hausdorff dimension $> n-4$.
Then there is a constant $\be>n-4$ so that $\be$-dimensional Hausdorff measure of $\mathcal{S}$ satisfies $\mathcal{H}^\be(\mathcal{S})>0$.
Let $\mathcal{H}^\be_\infty$ be a measure defined by
\begin{equation}\aligned\nonumber
\mathcal{H}^\be_\infty(E)=\omega_\be 2^{-\be}\inf\left\{\sum_{j=1}^\infty(\mathrm{diam} U_j)^\be\bigg|\, E\subset\bigcup_{j=1}^\infty U_j\subset\R^{n}\times\R^n\right\}
\endaligned
\end{equation}
for any set $E$ in $\R^{n}\times\R^n$, where $\omega_\be=\f{\pi^{\be/2}}{\G(\f \be2+1)}$, and $\G(r)=\int_0^\infty e^{-t}t^{r-1}dt$ is the gamma function for $0<r<\infty$.
From Lemma 11.2 in \cite{Gi}, $\mathcal{H}^\be(E)=0$ if and only if $\mathcal{H}^\be_\infty(E)=0$.
From the argument of Proposition 11.3 in \cite{Gi}, there are a point $\mathbf{q}\in \mathcal{S}$ and a sequence $r_j\rightarrow0$ as $j\rightarrow\infty$ such that
\begin{equation}\aligned
\mathcal{H}^\be_\infty\left(\mathcal{S}\cap \mathbf{B}_{r_j}(\mathbf{q})\right)>2^{-\be-1}\omega_\be r_j^\be.
\endaligned
\end{equation}
Up to translation, we assume $\mathbf{q}$ being the origin in $\R^n\times\R^n$.
Let $M_j=\f1{r_j}\left(M\cap \mathbf{B}_1\right)$, $\mathcal{S}_j=\f1{r_j}\left(\mathcal{S}\cap \mathbf{B}_{r_j}\right)$. Then
\begin{equation}\aligned
\mathcal{H}^\be_\infty\left(\mathcal{S}_j\cap \mathbf{B}_{1}\right)>2^{-\be-1}\omega_\be.
\endaligned
\end{equation}
Without loss of generality, we assume that $|M_j|$ converges to a tangent cone $M_*$ of $M$ in $\R^n\times\R^n$ in the varifold sense as $j\rightarrow\infty$.
From Lemma \ref{Multi1}, $M_*$ has multiplicity one everywhere on spt$M_*$.
Let $\mathcal{S}_*$ be the singular set of $M_*$.
If $y_j\in \mathcal{S}_j$ and $y_j\rightarrow y_*\in \mathrm{spt}M_*$, then
it's clear that $y_*$ is a singular point of $M_*$ by Allard's regularity theorem, which implies $\limsup_j\mathcal{S}_j\subset\mathcal{S}_*$. Analog to the proof of Lemma 11.5 in \cite{Gi}, we have $\mathcal{H}^\be_\infty\left(\mathcal{S}_{*}\cap \mathbf{B}_{1}(0)\right)>2^{-\be-1}\omega_\be$, and then
\begin{equation}\aligned
\mathcal{H}^\be\left(\mathcal{S}_{*}\cap \mathbf{B}_{1}\right)\ge\mathcal{H}^\be_\infty\left(\mathcal{S}_{*}\cap \mathbf{B}_{1}\right)>2^{-\be-1}\omega_\be.
\endaligned
\end{equation}
Let us continue the above procedure.
By the dimension reduction argument, there is an $n$-dimensional minimal cone $C\subset\R^{n}\times\R^n$,
such that for some integer $0<k\le3$,
$C$ is a trivial product of $\R^{n-k}$ and a $k$-dimensional regular but non-flat minimal cone $C_*$.
From Lemma \ref{Multi1} and the assumption $-K\le D^2u\le K$ a.e. on $B_R$, $C_*$ has multiplicity one and spt$C_*$ can be written as a graph over $\R^3$.
However, this contradicts to Theorem 5.4 in \cite{Fc} (see also Theorem 1.3 in \cite{Y1}). We complete the proof.
\end{proof}

From Allard's regularity theorem (see Theorem 24.2 in \cite{S} for instance), there is a positive constant $\tau_n>0$ depending only on $n$ such that
if $V$ is a multiplicity one stationary $n$-varifold in $\mathbf{B}_r(\mathbf{q})\subset\R^n\times\R^n$ with $\mathbf{q}\in\mathrm{spt} V$ and $\p(\mathrm{spt} V)\subset\p\mathbf{B}_r(\mathbf{q})$ such that the Radon measure $\mu_V$ associated to $V$ satisfies
\begin{equation}\aligned\label{DEFepn}
\mu_V(\mathbf{B}_r(\mathbf{q}))\le(1+\tau_n)\omega_nr^n,
\endaligned
\end{equation}
then spt$V$ is smooth in $\mathbf{B}_{r/2}(\mathbf{q})$, and the second fundamental form $B_V$ of spt$V\cap \mathbf{B}_{r/2}(\mathbf{q})$ satisfies
\begin{equation}\aligned\label{DEFepn*}
|B_V|\le \f1r\qquad\qquad \mathrm{on}\ \, \mathbf{B}_{r/2}(\mathbf{q})\cap \mathrm{spt}V.
\endaligned
\end{equation}

For any $C^2$ function $f$ on an open subset of $\R^n$, let $\overline{\la}_f(x)$ denote the largest eigenvalue of $D^2f(x)$, and $\underline{\la}_f(x)$ denote the smallest eigenvalue of $D^2f(x)$.
\begin{theorem}\label{BDSLK}
For any constant $K\ge1$, there is a constant $c_{n,K}>0$ depending only on $n,K$ such that if $u$ is a smooth solution to \eqref{SL} on $\R^n$ with the eigenvalues $\la_1,\cdots,\la_n$ of the Hessian $D^2u$ satisfying
\begin{equation}\aligned\label{Assump}
\la_i\la_j\ge-K\qquad\qquad on\ \R^n
\endaligned
\end{equation}
for all $i,j=1,\cdots,n$, and $u$ is not a quadratic polynomial, then the Hessian of $u$ satisfies $-c_{n,K}\le D^2u\le c_{n,K}$ on $\R^n$.
\end{theorem}
\begin{proof}
Let us prove it by contradiction. Let $K$ be a positive constant $\ge1$. Suppose that there is a sequence of smooth solutions $u_k$ to \eqref{SL} on $\R^n$ with the eigenvalues $\la_{1,k},\cdots,\la_{n,k}$ of $D^2u_k$ satisfying
\begin{equation}\aligned\label{Assumpk}
\la_{i,k}\la_{j,k}\ge-K\qquad\qquad on\ \R^n
\endaligned
\end{equation}
for all $i,j=1,\cdots,n$ and $k\ge1$ such that each $u_k$ is not a quadratic polynomial and $\lim_{k\rightarrow\infty}\sup_{\R^n}|D^2u_k|=\infty$.
Then there is a sequence of points $p_k\in\R^n$ such that $\overline{\la}_{u_k}(p_k)\rightarrow\infty$ as $k\rightarrow\infty$.
Let $\Si_k$ be the special Lagrangian graph in $\R^n\times\R^n$ with the graphic function $Du_k$.
Without loss of generality, we assume $Du_k=0$, then $\Si_k$ contains the origin in $\R^n\times\R^n$.
Since non-quadratic $u_k$ implies that $\Si_k$ is not flat, then any tangent cone of $\Si_k$ at infinity is not flat.
Let $\tau_n$ be the constant in \eqref{DEFepn}. We claim
\begin{equation}\aligned\label{Ratioge1epn}
\lim_{r\rightarrow\infty}\f1{\omega_nr^n}\mathcal{H}^n(\mathbf{B}_r\cap \Si_k)\ge1+\tau_n.
\endaligned
\end{equation}
Assume $\lim_{r\rightarrow\infty}\f1{\omega_nr^n}\mathcal{H}^n(\mathbf{B}_r\cap \Si_k)<1+\tau_n$.
Now let us deduce the contradiction. From monotonicity of the density $r^{-n}\mathcal{H}^n(\mathbf{B}_r\cap \Si_k)$, we have
$$\mathcal{H}^n(\mathbf{B}_r\cap \Si_k)<(1+\tau_n)\omega_nr^n$$
for all $r>0$. From \eqref{DEFepn}\eqref{DEFepn*}, we get $|B_{\Si_k}|\le \f1r$ on $\mathbf{B}_{r/2}\cap \Si_k$, where $B_{\Si_k}$ is the second fundamental form of $\Si_k$.
Letting $r\rightarrow\infty$ implies the flatness of $\Si_k$, which is a contradiction.
Hence the claim \eqref{Ratioge1epn} is true.
Then there are a sequence of numbers $r_k>0$ and a number $0<\tau_*<\tau_n$ such that
\begin{equation}\aligned\label{BrqMep}
\mathcal{H}^n(\mathbf{B}_{r_k}(\mathbf{p}_k)\cap \Si_k)=(1+\tau_*)\omega_nr_k^n
\endaligned
\end{equation}
with $\mathbf{p}_k=(p_k,Du(p_k))\in\R^n\times\R^n$. For each $k$, let
$$M_k=\f1{r_k}(\Si_k-\mathbf{p}_k)=\{r_k^{-1}(\mathbf{x}-\mathbf{p}_k)\in\R^n\times\R^n|\, \mathbf{x}\in \Si_k\},$$
which is a special Lagrangian graph through the origin. From \eqref{BrqMep},
we have
\begin{equation}\aligned\label{BrqMiep}
\mathcal{H}^n(\mathbf{B}_{1}\cap M_k)=(1+\tau_*)\omega_n.
\endaligned
\end{equation}
Denote $\hat{u}_k(x)=r_k^{-2}u_k(r_kx+p_k)$ for any $x\in\R^n$. Then $D\hat{u}_k$ is the graphic function of $M_k$.

Up to choose the subsequence, without loss of generality, we assume that the phase $\sum_i\arctan \la_{i,k}$ is a nonnegative constant for each $k$.
From Lemma \ref{mulowerbd} and the assumption \eqref{Assump}, we get
\begin{equation}\aligned\label{LowD2uK}
D^2\hat{u}_k\ge-2nK/\pi.
\endaligned
\end{equation}
Let $F_\La:\,(x,y)\rightarrow(\hat{x},\hat{y})$ be the isometric mapping from $\R^n\times\R^n$ into $\R^n\times\R^n$ defined as \eqref{labxby} with $\La=2nK/\pi$.
From \eqref{LowD2uK}, for each $k$ there is a smooth solution $\bar{u}_k$ to \eqref{SL} (with another phase different from the one for $D^2\hat{u}_k$) on $\R^n$ such that $F_\La(M_k)$ is the graph of $D\bar{u}_k$, i.e., $$F_\La(M_k)=\{(\bar{x},D\bar{u}_k(\bar{x}))\in\R^n\times\R^n|\ \bar{x}\in\R^n\}.$$
Put $M_{k,\La}=F_\La(M_k)$.
From \eqref{BOUNDD2u} and \eqref{LowD2uK}, we have
\begin{equation}\aligned\label{D2baruk******}
-\f{2\La^2+1}{\La}\le D^2\bar{u}_k\le2\La\qquad \mathrm{on}\ \ \R^n.
\endaligned
\end{equation}
From Lemma \ref{Multi1}, there are a $C^{1,1}$-function $\bar{u}_\infty$ with
\begin{equation}\aligned\label{ULD2baruinfty}
-\f{2\La^2+1}{\La}\le D^2\bar{u}_\infty\le2\La\qquad a.e.\ \mathrm{on}\ \ \R^n
\endaligned
\end{equation}
and a multiplicity one stationary $n$-varifold $V_{\infty,\La}$ such that
up to a choice of the subsequence $|M_{k,\La}|$ converges to $V_{\infty,\La}$ in the varifold sense with $\mathrm{spt}V_{\infty,\La}=\{(x,Du_\infty(x))\in\R^n\times\R^n|\, x\in \R^n\}$.
Denote $M_{\infty,\La}=\mathrm{spt}V_{\infty,\La}$. Noting that $M_{k,\La}$ has the uniformly gradient estimate in \eqref{D2baruk******} for each $k$. 
From \eqref{BrqMiep} and varifold convergence of $|M_{k,\La}|$, we get
\begin{equation}\aligned\label{B1Minftytau}
\mathcal{H}^n(\mathbf{B}_{1}\cap M_{\infty,\La})=(1+\tau_*)\omega_n.
\endaligned
\end{equation}
Combining \eqref{DEFepn}\eqref{DEFepn*}, we get $|B_{M_{\infty,\La}}|\le 1$ on $\mathbf{B}_{1/2}\cap M_{\infty,\La}$, where $B_{M_{\infty,\La}}$ is the second fundamental form of  $\mathbf{B}_{1/2}\cap M_{\infty,\La}$.
According to the isometric mapping $F_\La$, there is a countably $n$-rectifiable set $M_\infty$ in $\R^n\times\R^n$ such that $|M_k|$ converges to a stationary varifold $|M_\infty|$ in the varifold sense with $M_{\infty,\La}=F_\La(M_\infty)$,
and $|B_{M_{\infty}}|\le 1$ on $\mathbf{B}_{1/2}\cap M_{\infty}$, where $B_{M_{\infty}}$ is the second fundamental form of $\mathbf{B}_{1/2}\cap M_{\infty}$.

Let $v_k=\sqrt{\det(I+D^2\hat{u}_kD^2\hat{u}_k)}$. We see $v_{k}$ being a function on $M_{k}$ by identifying $v_{k}(x,D\hat{u}_k(x))=v_{k}(x)$.
Let $v_{k,\La}$ be a function on $F_\La(M_k)$ defined by
\begin{equation}\aligned\label{vkvkLa}
v_k(\mathbf{x})=v_{k,\La}(F_\La(\mathbf{x}))\qquad \mathrm{for\ any}\ \mathbf{x}\in M_k.
\endaligned
\end{equation}
Let $B_{M_{k}}$ denote the second fundamental form of $M_{k}$ in $\R^n\times\R^n$.
With \eqref{DEFepn}\eqref{DEFepn*}\eqref{BrqMiep}, we have $|B_{M_{k}}|\le 1$ on $\mathbf{B}_{1/2}\cap M_{k}$.
Combining \eqref{v-1nA} and Lemma \ref{ISOtrans}, we have
\begin{equation}\aligned\label{vLa-1nA}
&\De_{M_{k,\La}} v_{k,\La}^{-\f1n}\Big|_{F_\La(\mathbf{x})}=\De_{M_{k}}\left( v_{k,\La}^{-\f1n}\circ F_\La\right)(\mathbf{x})=\De_{M_k}v_k^{-\f1n}(\mathbf{x})\\
\le&\f{K-1}n v_{k}^{-\f1n}(\mathbf{x})|B_{M_{k}}|^2(\mathbf{x})\le\f{K-1}n v_{k}^{-\f1n}(\mathbf{x})=\f{K-1}n v_{k,\La}^{-\f1n}(F_\La(\mathbf{x}))
\endaligned
\end{equation}
for any $F_\La(\mathbf{x})\in\mathbf{B}_{1/2}\cap M_{k,\La}$.
From Theorem \ref{MVIsuper}, there are constants $\de_n\in(0,1]$ and $\th_{n,K}>0$ depending only on $n,K$ such that
\begin{equation}\aligned
\f1{\mathcal{H}^n\left(M_{k,\La}\cap \mathbf{B}_{1/4}\right)}\int_{M_{k,\La}\cap \mathbf{B}_{1/4}}v_{k,\La}^{-\f{\de_n}n}\le \th_{n,K}v_{k,\La}^{-\f{\de_n}n}(0).
\endaligned
\end{equation}
Since $F_\La$ is isometric with $M_{k,\La}=F_\La(M_k)$, then with \eqref{vkvkLa} the above inequality is equivalent to the following mean value inequality:
\begin{equation}\aligned\label{subv}
\f1{\mathcal{H}^n\left(M_k\cap \mathbf{B}_{1/4}\right)}\int_{M_k\cap \mathbf{B}_{1/4}}v_k^{-\f{\de_n}n}\le \th_{n,K}v_k^{-\f{\de_n}n}(0).
\endaligned
\end{equation}
Note that $\lim_{k\rightarrow\infty}\overline{\la}_{u_k}(p_k)=\infty$ implies $v_k(0)\rightarrow\infty$. Combining \eqref{subv} and $|B_{M_k}|\le1$ on $\mathbf{B}_{1/2}\cap M_k$, we conclude that
\begin{equation}\aligned\label{Milocinfty}
\lim_{k\rightarrow\infty}\inf_{M_k\cap \mathbf{B}_{1/8}}v_k=\infty.
\endaligned
\end{equation}
Let $\mathcal{S}_{M_\infty}$ be the singular set of $M_\infty$, then $\mathcal{S}_{M_{\infty,\La}}=F_\La(\mathcal{S}_{M_\infty})$.
Note that $\mathcal{S}_{M_\infty}$ is closed in $M_\infty$, $|M_k|$ converges to the $n$-rectifiable stationary varifold $|M_\infty|$ in the varifold sense, and $|M_\infty|$ has multiplicity one everywhere on $M_\infty$.
Then from Allard's regularity theorem for any $\mathbf{x}\in M_\infty\setminus \mathcal{S}_{M_\infty}$, there is a constant $r_{\mathbf{x}}>0$ such that $M_k\cap B_{r_{\mathbf{x}}}(\mathbf{x})$ converges to $M_\infty\cap B_{r_{\mathbf{x}}}(\mathbf{x})$ smoothly.
Since $\mathcal{S}_{M_\infty}$ has codimension 4 at least by Lemma \ref{SingHD}, then $M_\infty\setminus \mathcal{S}_{M_\infty}$ is connected.
Combining \eqref{Milocinfty} and the mean value inequality for $v_k$ like \eqref{subv}, for any compact set $\Om$ with $\Om\cap \mathcal{S}_{M_\infty}=\emptyset$, we have
\begin{equation}\aligned\label{Miinfty}
\lim_{k\rightarrow\infty}\inf_{M_k\cap \Om}v_k=\infty.
\endaligned
\end{equation}
Then combining \eqref{Miinfty} and the assumption \eqref{Assump}, we have
\begin{equation}\aligned\label{minuiconv}
\liminf_{k\rightarrow\infty}\inf_{M_k\cap \Om}\underline{\la}_{u_k}\ge0.
\endaligned
\end{equation}
From \eqref{RLubu}, we have
\begin{equation}\aligned\label{LUbaru}
-\f{1}{2\La}\le \underline{\la}_{\bar{u}_\infty}\le \overline{\la}_{\bar{u}_\infty}=2\La
\endaligned
\end{equation}
on the set where $\bar{u}_\infty$ is $C^2$.

If $\mathcal{S}_{M_{\infty,\La}}$ is empty, then by a standard argument (see Yuan \cite{Y1,Y3}) \eqref{LUbaru} on $\R^n$ implies the flatness of $M_{\infty,\La}$, i.e., $\bar{u}_\infty$ is a quadratic polynomial.
However, this violates \eqref{B1Minftytau}.
Hence $\mathcal{S}_{M_{\infty,\La}}\neq\emptyset$.
Now we blow $M_{\infty,\La}$ up at a point in $\mathcal{S}_{M_{\infty,\La}}$.
By (Federer's) dimension reduction argument, we get a multiplicity one minimal cone $C_*$, which is a trivial product of a $l$-dimensional nonflat regular minimal cone $C^*$ and a Euclidean factor $\R^{n-l}(4\le l\le n)$.
With Lemma \ref{Multi1}, there is a sequence of manifolds $\widehat{M}_k$ (obtained from $M_{\infty,\La}$ by scaling and translation) such that
$\left|\widehat{M}_k\right|$ converges in the varifold sense to $C_*$.
Moreover, from \eqref{ULD2baruinfty} there is a $C^{1,1}$ function $w^*$ on $\R^n$ with
\begin{equation}\aligned\label{D2u*}
-\f{2\La^2+1}{\La}\le D^2w^*\le2\La\qquad a.e.\ \mathrm{on}\ \ \R^n
\endaligned
\end{equation}
such that spt$C_*$ is a graph over $\R^n$ in $\R^n\times\R^n$ with the graphic function $Dw^*$.
Note that $w^*$ is not $C^2$ on the set $\{0^l\}\times\R^{n-l}=\{(0,\cdots,0,x_{l+1},\cdots,x_n)\in\R^n|\, (x_{l+1},\cdots,x_n)\in\R^{n-l}\}$.
From \eqref{LUbaru}, we get
\begin{equation}\aligned\label{LUw*}
-\f{1}{2\La}\le \underline{\la}_{w^*}\le \overline{\la}_{w^*}=2\La
\endaligned
\end{equation}
on $\R^n\setminus(\{0^l\}\times\R^{n-l})$.
Hence, without loss of generality, for each $j=1,\cdots,n$, $\p_jw^*$ has the decomposition as follows:
$$\p_jw^*(x_1,\cdots,x_n)=\f{\p}{\p x_j}w^*(x_1,\cdots,x_n)=\phi_j(x_1,\cdots,x_l)+\sum_{k=l+1}^nc_{jk}x_k$$
for some function $\phi_j$ and some constant matrix $(c_{jk})_{n\times(n-l)}$ with $k=l+1,\cdots,n$ and $j=1,\cdots,n$.
By choosing the coordinate system of $x_1,\cdots,x_n$, we can assume
\begin{equation}\aligned\label{pjw*phijcj}
\p_jw^*(x_1,\cdots,x_n)=\phi_j(x_1,\cdots,x_l)+c_{j}x_j
\endaligned
\end{equation}
for some constant vector $(c_1,\cdots,c_n)$ with $c_1=\cdots=c_l=0$.
Outside $\{0^l\}\times\R^{n-l}$,
for $i=1,\cdots,l$ and $j=l+1,\cdots,n$, we have
\begin{equation}\aligned
\p_{ij}w^*=\p_{ji}w^*=0=\p_i\phi_j(x_1,\cdots,x_l),
\endaligned
\end{equation}
which implies $\phi_j(x_1,\cdots,x_l)=0$ on $\R^l$ for $j=l+1,\cdots,n$ as spt$C_*$ contains the origin. We consider a $C^{1,1}$ function $\Phi^*$ on $\R^n$ defined by
\begin{equation}\aligned\label{Phi***}
\Phi^*=w^*-\f12\sum_{j=l+1}^nc_{j}x_j^2.
\endaligned
\end{equation}
Then with \eqref{pjw*phijcj} we have $\p_j\Phi^*=\p_jw^*-c_jx_j=0$ for any $j=l+1,\cdots,n$.
In other words, $\Phi^*$ is a function depending only on $x_1,\cdots,x_l$. Hence we can define a function $\Phi$ on $\R^l$ by $\Phi(x_1,\cdots,x_l)=\Phi^*(x_1,\cdots,x_l,0,\cdots,0)$ such that
$$\mathrm{spt}C^*=\{(x,D\Phi(x))\in\R^l\times\R^l|\, x\in\R^l\}.$$
Let $\mu_1\ge\cdots\ge\mu_l$ be the eigenvalues of $D^2\Phi$.
Then $D^2\Phi^*$ has eigenvalues $\mu_1,\cdots,\mu_l$ and $c_{l+1},\cdots,c_n$ with
\begin{equation}\aligned
\sum_{i=1}^l\arctan\mu_i+\sum_{i=l+1}^n\arctan c_j=\Th.
\endaligned
\end{equation}
From the calibration $Re(e^{-\sqrt{-1}(\Th-\sum_{i=l+1}^n\arctan c_j)}dz_1\wedge \cdots\wedge dz_l)$, the Lagrangian graph spt$C^*$ is minimizing in $\R^l\times\R^l$.
Note that spt$C^*$ has only one singularity at the origin. From \eqref{LUw*} and \eqref{Phi***}, $\Phi(x_1,\cdots,x_l)+\f12\sum_{j=l+1}^nc_{j}x_j^2$ has eigenvalues between $-\f{1}{2\La}$ and $2\La$, which implies
\begin{equation}\aligned
2\La\ge\mu_1\ge\cdots\ge\mu_l\ge-\f{1}{2\La}
\endaligned
\end{equation}
on $\R^l\setminus\{0^l\}$. By the maximum principle argument for $\sqrt{\mathrm{det}(I+(D^2\Phi)^2)}$ (see Yuan \cite{Y1,Y3} for instance), we get the flatness of spt$C^*$, which is a contradiction. This completes the proof.
\end{proof}

Combining Warren-Yuan's argument in \cite{WmY2}, we have the following Liouville type theorem.
\begin{theorem}\label{mainLiou}
There exists a constant $\ep_n\in(0,1)$ such that if $u$ is a smooth solution to the special Lagrangian equation \eqref{SL} on $\R^n$ with eigenvalues $\la_1(x),\cdots,\la_n(x)$ of the Hessian $D^2u(x)$ satisfying
\begin{equation}\aligned\label{1+epnlailaj}
3(1+\ep_n)+(1+\ep_n)\la_i^2(x)+2\la_i(x)\la_j(x)\ge0
\endaligned
\end{equation}
for all $i,j=1,\cdots,n$ and $x\in\R^n$, then $u$ must be a quadratic polynomial.
\end{theorem}
\begin{proof}
Let us prove it by contradiction. Suppose Theorem \ref{mainLiou} fails. Then there is a sequence of smooth solutions $u_k$ to \eqref{SL} on $\R^n$ with eigenvalues $\la_{1,k}(x),\cdots,\la_{n,k}(x)$ of $D^2u_k(x)$ satisfying
\begin{equation}\aligned\label{1klailaj}
3\left(1+\f1k\right)+\left(1+\f1k\right)\la_{i,k}^2(x)+2\la_{i,k}(x)\la_{j,k}(x)\ge0
\endaligned
\end{equation}
for all $i,j=1,\cdots,n$, $k\ge2$ and $x\in\R^n$, and each $u_k$ is not a quadratic polynomial.
The inequality \eqref{1klailaj} implies
\begin{equation}\aligned
\left(1-\f1k\right)\la_{i,k}(x)\la_{j,k}(x)\ge-3\left(1+\f1k\right)
\endaligned
\end{equation}
for all $i,j=1,\cdots,n$, $k\ge2$ and $x\in\R^n$.
From Theorem \ref{BDSLK}, $D^2u_k$ is uniformly bounded on $\R^n$ by a constant $c_n$ depending only on $n$, i.e., $-c_n\le D^2u_k\le c_n$.
Let $M_k$ denote the Lagrangian graph with the graphic function $Du_k$ for each $k$. From Allard' regularity theorem (see also \eqref{Ratioge1epn}), there exists a sequence $r_k\rightarrow\infty$ such that
\begin{equation}\aligned
\liminf_{k\rightarrow\infty}r_k^{-n}\mathcal{H}^n(\mathbf{B}_{r_k}\cap M_k)>\omega_n.
\endaligned
\end{equation}
Let $\widehat{M}_k=\f1{r_k}M_k$, and $\hat{u}_k=r_k^{-2}u_k(r_k\cdot)$. Then
\begin{equation}\aligned\label{widehatMk}
\liminf_{k\rightarrow\infty}\mathcal{H}^n(\mathbf{B}_1\cap \widehat{M}_k)>\omega_n.
\endaligned
\end{equation}
From Lemma \ref{Multi1} and $-c_n\le D^2u_k\le c_n$,
there are a $C^{1,1}$-function $\hat{u}_\infty$ on $\R^n$ with
\begin{equation}\aligned\label{D2uhinftycn}
-c_n\le D^2\hat{u}_\infty\le c_n\qquad a.e.\ \mathrm{on}\ \ \R^n
\endaligned
\end{equation}
and a multiplicity one stationary $n$-varifold $V_\infty$ such that
up to a choice of the subsequence $|\widehat{M}_k|$ converges to $V_\infty$ in the varifold sense with $\mathrm{spt}V_\infty=\{(x,D\hat{u}_\infty(x))\in\R^n\times\R^n|\, x\in \R^n\}$.

From Lemma \ref{SingHD},
there is a closed set $\mathcal{S}$ in $\R^n$ with Hausdorff dimension of $\mathcal{S}\le n-4$ such that $\hat{u}_\infty$ is smooth on $\R^n\setminus\mathcal{S}$. Let $\hat{\la}_{1,k},\cdots,\hat{\la}_{n,k}$ denote the eigenvalues of $D^2\hat{u}_k$ with $\hat{\la}_{1,k}\ge\cdots\ge\hat{\la}_{n,k}$. From $-c_n\le D^2u_k\le c_n$, we have $c_n\ge\hat{\la}_{1,k}\ge\cdots\ge\hat{\la}_{n,k}\ge-c_n$.
For any compact set $K$ in $\R^n$ with $K\cap\mathcal{S}=\emptyset$, from Lemma \ref{Multi1} we can assume that $\hat{u}_k$ converges to $\hat{u}_\infty$ smoothly on $K$ up to a choice of the subsequence. 
In particular, for any $x\in K$, $D^2\hat{u}_k(x)\to D^2\hat{u}_\infty(x)$. For each $k\ge1$, let $(\xi_{1,k},\cdots,\xi_{n,k})$ be an orthonormal $(n\times n)$-matrix with unit vectors $\xi_{1,k},\cdots,\xi_{n,k}$ so that
\begin{equation}\aligned\label{hatlaikxD2hu}
\hat{\la}_{i,k}(x)=\left\lan \xi_{i,k},D^2\hat{u}_k(x)\xi_{i,k}\right\ran\qquad \qquad \mathrm{for\ each}\ i.
\endaligned
\end{equation}
Up to a choice of the subsequences, we assume $\xi_{i,k}\to\xi_i$ for some vector $\xi_i$ for each $i$. Then $(\xi_{1},\cdots,\xi_{n})$ forms an orthonormal $(n\times n)$-matrix with $|\xi_1|=\cdots=|\xi_n|=1$. From \eqref{hatlaikxD2hu},
we have
\begin{equation}\aligned
\lim_{k\to\infty}\hat{\la}_{i,k}(x)=\lim_{k\to\infty}\left\lan\xi_{i,k},D^2\hat{u}_k(x)\xi_{i,k}\right\ran=\lim_{k\to\infty}\left\lan\xi_{i},D^2\hat{u}_\infty(x)\xi_{i}\right\ran.
\endaligned
\end{equation}
Denote $\hat{\la}_{i,x}=\lim_{k\to\infty}\hat{\la}_{i,k}(x)$.
Then $\hat{\la}_{i,x}$ is the eigenvalue of $D^2\hat{u}_\infty(x)$ with the eigenvector $\xi_i$.
From \eqref{1klailaj}, it follows that $3+\hat{\la}_{i,x}^2+2\hat{\la}_{i,x}\hat{\la}_{j,x}\ge0$ for each $i,j=1,\cdots,n$.
Let $\hat{\la}_1,\cdots,\hat{\la}_n$ be the eigenvalues of $D^2\hat{u}_\infty$ on $\R^n\setminus\mathcal{S}$ with $\hat{\la}_{1}\ge\cdots\ge\hat{\la}_{n}$.
The above argument gives
\begin{equation}\aligned
3+\hat{\la}_i^2+2\hat{\la}_i\hat{\la}_j\ge0\qquad \ \ \mathrm{on}\ \R^n\setminus\mathcal{S}
\endaligned
\end{equation}
for all $i,j=1,\cdots,n$. Together with \eqref{D2uhinftycn},
now we follow the proof of Theorem 3.1 in \cite{WmY2} including Proposition 3.1 in \cite{WmY2}, and get the flatness of spt$V_\infty$. However, this contradicts to \eqref{widehatMk} with the help of Lemma \ref{Multi1}. We complete the proof.
\end{proof}
\textbf{Remark.}
Obviously if we assume
\begin{equation}\aligned\label{epnlailaj}
\inf_{i,j=1,\cdots,n}\la_i(x)\la_j(x)\ge-\f32(1+\ep_n)
\endaligned
\end{equation}
for any $x\in\R^n$, or
\begin{equation}\aligned\label{epnlai*}
-\sqrt{3}(1+\ep_n)\le D^2u\le\sqrt{3}(1+\ep_n)
\endaligned
\end{equation}
on $\R^n$, then \eqref{1+epnlailaj} holds true. Namely, any smooth solution $u$ to \eqref{SL} on $\R^n$ satisfying \eqref{epnlailaj} for any $x\in\R^n$ or \eqref{epnlai*} on $\R^n$ must be a quadratic polynomial.

Using the above Liouville type theorem, we can get an interior curvature estimate for special Lagrangian graphs, which is key for Hessian estimates of solutions to special Lagrangian equations in the following section.
\begin{corollary}\label{CurvEst}
Let $u$ be a smooth solution to the special Lagrangian equation \eqref{SL} on $B_2\subset\R^n$ with $Du(0)=0$, and the eigenvalues $\la_1(x),\cdots,\la_n(x)$ of $D^2u(x)$ satisfies
\eqref{1+epnlailaj} for all $i,j$ and $x\in B_2$. Let $M$ be the special Lagrangian graph over $B_2$ with the graphic function $Du$.
Then there is a constant $c_n>0$ depending only on $n$ such that $|B_M|\le c_n$ on $M\cap \mathbf{B}_1$, where $B_M$ is the second fundamental form of $M$.
\end{corollary}
\begin{proof}
Let us prove it by contradiction. Suppose that there is a sequence of smooth solution $u_k$ to the special Lagrangian equation \eqref{SL} on $B_2\subset\R^n$ with $Du_k(0)=0$
such that the eigenvalues $\la_{1,k},\cdots,\la_{n,k}$ of $D^2u_k$ satisfies
\begin{equation}\aligned\label{1+epnlailaj***}
3(1+\ep_n)+(1+\ep_n)\la_{i,k}^2(x)+2\la_{i,k}(x)\la_{j,k}(x)\ge0
\endaligned
\end{equation}
for all $i,j=1,\cdots,n$, $k\ge1$ and $x\in B_2$, and the special Lagrangian graph $M_k$ of the graphic function $u_k$ satisfies
\begin{equation}\aligned
\lim_{k\rightarrow\infty}|B_{M_k}|(z_k)=\infty
\endaligned
\end{equation}
for some sequence of points $z_k\in M_k\cap \mathbf{B}_1$. Here, $B_{M_k}$ is the second fundamental form of $M_k$.

Then there exists a sequence of points $\mathbf{q}_k\in \mathbf{B}_{\f32}$ such that
\begin{equation}\aligned\label{rkgotoinfty}
r_k\triangleq\left(\f32-|\mathbf{q}_k|\right)|B_{M_k}|(\mathbf{q}_k)=\sup_{\mathbf{B}_{\f32}\cap M_k}\left(\f32-|\mathbf{x}|\right)|B_{M_k}|(\mathbf{x}) \rightarrow\infty
\endaligned
\end{equation}
as $k\rightarrow\infty$. Denote $\mathbf{q}_k=(q_k,Du_k(q_k))$.
Put $\tau_k=\f32-|\mathbf{q}_k|>0$, then $\mathbf{B}_{\tau_k}(\mathbf{q}_k)\subset\mathbf{B}_{\f32}$. Let $R_k=2r_k/\tau_k$, $\Si_k$ be a scaling of a part of $M_k$ through the origin defined by
$$\Si_k=\{R_k(\mathbf{x}-\mathbf{q}_k)\in\R^n\times\R^n|\, \mathbf{x}\in M_k\cap \mathbf{B}_{\tau_k/2}(\mathbf{q}_k)\},$$
and
$$\hat{u}_k(x)=R_k^2u_k(R_k^{-1}x+q_k)-R_kx\cdot Du_k(q_k)\qquad \mathrm{for\ any}\ x\in\Pi(M_k\cap \mathbf{B}_{\tau_k/2}(\mathbf{q}_k)),$$
then $\Si_k$ is a special Lagrangian graph in $\mathbf{B}_{r_k}$ with the graphic function $D\hat{u}_k$ and $\p\Si_k\cap\mathbf{B}_{r_k}=\emptyset$.

Let $B_{\Si_k}$ be the second fundamental form of $\Si_k$ in $\R^n\times\R^n$.
Since $\f{\tau_k}2\le\f32-|x|$ for all $x\in \mathbf{B}_{\f{\tau_k}2}(\mathbf{q}_k)$, then by the definition of $r_k$ we have
\begin{equation}\aligned\label{BdSii}
\sup_{\Si_k}|B_{\Si_k}|=&\f1{R_k}\sup_{\mathbf{B}_{\f{\tau_k}2}(\mathbf{q}_k)\cap M_k}|B_{M_k}|
\le\f1{R_k}\f{2}{\tau_k}\sup_{\mathbf{B}_{\f{\tau_k}2}(\mathbf{q}_k)\cap M_k}\left(\f32-|\mathbf{x}|\right)|B_{M_k}|(\mathbf{x})\\
\le&\f{2}{R_k\tau_k}\sup_{\mathbf{B}_{\f32}\cap M_k}\left(\f32-|\mathbf{x}|\right)|B_{M_k}|(\mathbf{x})=\f{2r_k}{R_k\tau_k}=1,
\endaligned
\end{equation}
and
\begin{equation}\aligned\label{BdSii0}
|B_{\Si_k}|(0)=\f1{R_k}|B_{M_k}|(\mathbf{q}_k)
=\f1{R_k}\f{1}{\tau_k}\sup_{\mathbf{B}_{\f32}\cap M_k}\left(\f32-|\mathbf{x}|\right)|B_{M_k}|(\mathbf{x})=\f{r_k}{R_k\tau_k}=\f12.
\endaligned
\end{equation}

From \eqref{1+epnlailaj***}, there holds
\begin{equation}\aligned
3(1+\ep_n)+(1+\ep_n)\min\{\la_{i,k}^2(x),\la_{j,k}^2(x)\}+2\la_{i,k}(x)\la_{j,k}(x)\ge0
\endaligned
\end{equation}
for all $i,j=1,\cdots,n$ and $x\in B_2$, which implies
\begin{equation}\aligned\label{3*lailaj*}
3\f{1+\ep_n}{1-\ep_n}+\la_{i,k}(x)\la_{j,k}(x)\ge0.
\endaligned
\end{equation}
Let $\hat{\la}_{1,k}(x),\cdots,\hat{\la}_{n,k}(x)$ be the eigenvalues of $D^2\hat{u}_k(x)$. By the definition of $\hat{u}_k$, the above inequality implies
\begin{equation}\aligned\label{3lailaj*}
3\f{1+\ep_n}{1-\ep_n}+\hat{\la}_{i,k}(x)\hat{\la}_{j,k}(x)\ge0
\endaligned
\end{equation}
for each integer $k\ge1$, $i,j\in\{1,\cdots,n\}$, $x\in\Pi(\Si_k\cap\mathbf{B}_{r_k})$.
With Lemma \ref{mulowerbd} and \eqref{3lailaj*}, we get $\hat{\la}_{k,i}\ge-\La$ with $\La=\f{6n(1+\ep_n)}{\pi(1-\ep_n)}$ on $\Pi(\Si_k\cap\mathbf{B}_{r_k})$.

Let $F_\La$ be the Lewy rotation defined in \eqref{labxby}, then $F_\La(\Si_k)\subset\mathbf{B}_{r_k}$ through the origin with $\p(F_\La(\Si_k))\cap\mathbf{B}_{r_k}=\emptyset$  by the definition of $\Si_k$.
From \eqref{Dux2x1}, $F_\La(\Si_k)$ is a graph over $\Pi(F_\La(\Si_k))$ in $\mathbf{B}_{r_k}$.
From \eqref{yibxj}, for each $k$ there are $n$ smooth functions $(w_{k,1},\cdots,w_{k,n})$ on $\Pi(F_\La(\Si_k))$ with $w_{k,i}(0)=0$ and $\p_{i}w_{k,j}=\p_{j}w_{k,i}$ for each $i,j=1,\cdots,n$
such that $F_\La(\Si_k)$ is the graph of $(w_{k,1},\cdots,w_{k,n})$. Here, $\p_1,\cdots,\p_n$ is a standard orthonormal basis of $\R^n$.
A similar argument of \eqref{RLubu}\eqref{BOUNDD2u}, all the eigenvalues of the symmetric matrix $(\p_{i}w_{k,j})$ are between $-\f{2\La^2+1}\La$ and $2\La$, i.e.,
\begin{equation}\aligned\label{wkijbound}
-\f{2\La^2+1}\La\le(\p_{i}w_{k,j})\le2\La\qquad \mathrm{on}\ \Pi(F_\La(\Si_k)).
\endaligned
\end{equation}

Let $B_{\r_k}$ be the largest ball centered at the origin in $\Pi(F_\La(\Si_k))$. Denote $|w_{k,*}|=\sqrt{\sum_{i=1}^nw_{k,i}^2}$. Then from \eqref{wkijbound} for any $y\in\p B_{\r_k}$
\begin{equation}\aligned
|w_{k,*}|(y)=\int_0^1\f{\p}{\p t}|w_{k,*}|(ty)dt=\sum_{i=1}^n\int_0^1\f{w_{k,i}y\cdot Dw_{k,i}}{|w_{k,*}|}dt\le(2\La+1)|y|=(2\La+1)\r_k.
\endaligned
\end{equation}
Combining $\r_k^2+|w_{k,*}|^2(y)=r_k^2$, we get $\r_k\ge\f{r_k}{2(\La+1)}$. In particular, $\lim_{k\rightarrow\infty}\r_k=\infty$ from \eqref{rkgotoinfty}.
From Frobenius' theorem (see Lemma 7.2.11 in \cite{X} for instance), there is a function $\bar{u}_k$ on $B_{\r_k}$ with $D\bar{u}_k(0)=(0,\cdots,0)$ such that $D\bar{u}_k=(w_{k,1},\cdots,w_{k,n})$ on $B_{\r_k}$.
The inequality \eqref{wkijbound} implies
\begin{equation}\aligned\label{barukbound}
-\f{2\La^2+1}\La\le D^2\bar{u}_k\le2\La\qquad \mathrm{on}\ B_{\r_k}.
\endaligned
\end{equation}
From \eqref{BdSii} and the isometric mapping $F_\La$, we get $\sup_{F_\La(\Si_k)}|B_{F_\La(\Si_k)}|\le1$,
where $B_{F_\La(\Si_k)}$ is the second fundamental form of $F_\La(\Si_k)$ in $\R^n\times\R^n$.
Hence, $|D^3\bar{u}_{k}|$ is uniformly bounded on any compact set of $\R^n$ for all suitable large $k$.
By Schauder's theory of elliptic equations, there is a smooth function $\bar{u}_\infty$ on $\R^n$
such that up to a choice of the subsequence, $\bar{u}_{k}$ converges smoothly to $\bar{u}_\infty$ on compact sets of $\R^n$, and $\Si_{\infty,\La}\triangleq\{(x,D\bar{u}_\infty(x))\in\R^n\times\R^n|\ x\in\R^n\}$ is a smooth minimal graph over $\R^n$.
Hence, there is a smooth minimal submanifold $\Si_\infty$ with $F_\La(\Si_\infty)=\Si_{\infty,\La}$ such that $\Si_k\cap E$ converges smoothly to $\Si_\infty\cap E$ for any compact set $E\subset\R^n\times\R^n$.

Now let us finish the proof by dividing into two cases.
\begin{itemize}
  \item Case 1. There are a point $\mathbf{x}_\infty\in\Si_\infty$ and a sequence $\mathbf{x}_k\in\Si_k$ with $\mathbf{x}_k\rightarrow\mathbf{x}_\infty$ such that there holds
$|D^2\hat{u}_k|(x_k)\rightarrow\infty$ with $\mathbf{x}_k=(x_k,D\hat{u}_k(x_k))$.
From the proof of Theorem \ref{BDSLK}, the mean value inequality for the function $(\det(I+D^2\hat{u}_kD^2\hat{u}_k))^{-\f1{2n}}$ implies that for any sequence $\mathbf{y}_k\in\Si_k$ with $\mathbf{y}_k\rightarrow\mathbf{y}_\infty$ and $\limsup_k|\mathbf{y}_k|<\infty$, there holds
$|D^2\hat{u}_k|(y_k)\rightarrow\infty$ with $\mathbf{y}_k=(y_k,D\hat{u}_k(y_k))$.
By following the argument of the proof of Theorem \ref{BDSLK} on the part of the singular set of $M_\infty$, we get the flatness of $\Si_\infty$ from \eqref{3lailaj*}. However, this contradicts to \eqref{BdSii0}.

  \item Case 2. For any point $\mathbf{x}_\infty\in\Si_\infty$ and any sequence $\mathbf{x}_k=(x_k,D\hat{u}_k(x_k))\in\Si_k$ with $\mathbf{x}_k\rightarrow\mathbf{x}_\infty$, there holds $\limsup_{k\rightarrow\infty}|D^2\hat{u}_k|(x_k)<\infty$. With \eqref{BdSii} we get that the smooth minimal submanifold $\Si_\infty$ is a speical Lagrangian graph over $\R^n$ in $\R^n\times\R^n$. Let $u_\infty$ be a smooth function such that $Du_\infty$ is the graphic function of $\Si_\infty$. Since $\Si_k\cap E$ converges smoothly to $\Si_\infty\cap E$ for any compact set $E\subset\R^n\times\R^n$, then \eqref{3lailaj*} implies
\begin{equation}\aligned
3\f{1+\ep_n}{1-\ep_n}+\la_{i,\infty}(x)\la_{j,\infty}(x)\ge0,
\endaligned
\end{equation}
for each $i,j\in\{1,\cdots,n\}$, $x\in\R^n$, where $\la_{1,\infty},\cdots,\la_{n,\infty}$ are the eigenvalues of $D^2u_\infty$.
      From Theorem \ref{mainLiou}, we get the flatness of $\Si_\infty$, which contradicts to \eqref{BdSii0}.
\end{itemize}
\end{proof}

\section{Hessian estimates for special Lagrangian equations}

In this section, we use superharmonic functions on special Lagrangian graphs to derive Hessian estimates for the solutions to special Lagrangian equations.
\begin{theorem}\label{GraSL}
Let $u$ be a smooth solution to the special Lagrangian equation \eqref{SL} on $B_R\subset\R^n$ with the eigenvalues $\la_1(x),\cdots,\la_n(x)$ of the Hessian $D^2u(x)$ satisfying
\eqref{1+epnlailaj} for all $i,j$ and $x\in B_R$.
Then there is a constant $C_{n}>0$ depending only on $n$ such that
\begin{equation}\aligned
|D^2u(0)|\le C_{n} \mathrm{exp}\left(C_{n}\f{\max_{B_R}|Du-Du(0)|^n}{R^n}\right).
\endaligned
\end{equation}
\end{theorem}
\begin{proof}
By scaling, we only need to show the case of $R=3$. By considering $u-Du(0)\cdot x$, we can assume $Du(0)=0$.
Let $M_r=\{(x,Du)\in\R^n\times\R^n|\ x\in B_r\}$ for $r\in(0,3]$, and $M=M_3$ for short.
We consider the mapping $F_\La:\,(x,y)\rightarrow(\hat{x},\hat{y})$ as \eqref{labxby} with $\La=\f{6n(1+\ep_n)}{\pi(1-\ep_n)}$. Let $\bar{x},\bar{y}$ be mappings defined in \eqref{barxbaryx},
then
\begin{equation}\aligned\label{barxMr**}
\bar{x}(M_r)=\left\{\f1{\sqrt{4\La^2+1}}(2\La x+Du(x))\in\R^n\Big|\,x\in B_r\right\}.
\endaligned
\end{equation}
From Lemma \ref{mulowerbd} and \eqref{3lailaj*}, the function $\tilde{u}\triangleq\f1{\sqrt{4\La^2+1}}(u(x)+\La|x|^2)$ is convex with 
$D^2\tilde{u}\ge\La$ on $B_3$.
Since $\bar{x}: B_3\to D\tilde{u}(B_3)$ is injective from \eqref{Dux2x1}, and $\mathrm{det}J>0$ from \eqref{DEFJ}, then 
$\bar{x}(M)=D\tilde{u}(B_3)$ is simply connected.
Therefore, $F_\La(M)$ can be written as a graph over $\bar{x}(M)$ with the graphic function $D\bar{u}$ for some solution $\bar{u}$ to \eqref{SL}.
From \eqref{BOUNDD2u}, one has
\begin{equation}\aligned\label{La111222333}
-\f{2\La^2+1}{\La}\le D^2\bar{u}\le2\La\qquad \mathrm{on}\ \bar{x}(M).
\endaligned
\end{equation}

For any $t>0$, let $(t\Z)^n$ denote the lattice in $\R^n$ defined by
$$\{(x_1,\cdots,x_n)\in\R^n|\, t^{-1}x_i\in\Z\ \mathrm{for \ each}\ i\}.$$
Let $Q=(t_*\Z)^n\cap\bar{x}(\overline{M_1})$ with $t_*=\La/\sqrt{\La^2+(2\La^2+1)^2}$.
For any distinct $q_1,q_2\in Q$ with $|q_1-q_2|=t_*$, let $p_1,p_2\in\bar{x}(\overline{M_1})$ satisfy $q_i=\bar{x}(p_i)$ for $i=1,2$. From \eqref{x2x1},
\begin{equation}\aligned\label{p1p2La}
\lan p_1-p_2,\bar{x}(p_1)-\bar{x}(p_2)\ran=\f{\lan p_1-p_2,2\La p_1+Du(p_1)-2\La p_2-Du(p_2)\ran}{\sqrt{4\La^2+1}}\ge\f{\La|p_1-p_2|^2}{\sqrt{4\La^2+1}}.
\endaligned
\end{equation}
Then combining the definitions of $\bar{x},\bar{y}$ in \eqref{barxbaryx}, and \eqref{Dux2x1}\eqref{p1p2La}, one has
\begin{equation}\aligned
&|D\bar{u}(q_1)-D\bar{u}(q_2)|^2=|\bar{y}(p_1)-\bar{y}(p_2)|^2\\
=&\f1{4\La^2+1}\left|2\La\sqrt{4\La^2+1}(\bar{x}(p_1)-\bar{x}(p_2))-(4\La^2+1)(p_1-p_2)\right|^2\\
\le&4\La^2\left|\bar{x}(p_1)-\bar{x}(p_2)\right|^2-4\La^2|p_1-p_2|^2+(4\La^2+1)|p_1-p_2|^2\\
\le&4\La^2\left|\bar{x}(p_1)-\bar{x}(p_2)\right|^2+\f{4\La^2+1}{\La^2}\left|\bar{x}(p_1)-\bar{x}(p_2)\right|^2=\f{(2\La^2+1)^2}{\La^2}t_*^2.
\endaligned
\end{equation}
By the definition of $t_*$, we have
\begin{equation}\aligned\label{q1q2phiq1phiq2}
|q_1-q_2|^2+|D\bar{u}(q_1)-D\bar{u}(q_2)|^2\le t_*^2\left(1+\f{(2\La^2+1)^2}{\La^2}\right)=1.
\endaligned
\end{equation}
Put
\begin{equation}\aligned
\mathbf{Q}=\{(q,D\bar{u}(q))\in F_\La(\overline{M_1})|\, q\in Q\}.
\endaligned
\end{equation}
Then from \eqref{q1q2phiq1phiq2} we have
\begin{equation}\aligned\label{FMcover}
F_\La(M_1)\subset \bigcup_{\mathbf{q}\in \mathbf{Q}}\mathbf{B}_{1}(\mathbf{q}).
\endaligned
\end{equation}

Let $v=\sqrt{\det(I+D^2uD^2u)}$, and we see $v$ being a function on $M$ by identifying $v(x,Du(x))=v(x)$.
Let $v_{\La}$ be a function on $F_\La(M)$ defined by
\begin{equation}\aligned\label{vkvkLa*}
v(\mathbf{x})=v_{\La}(F_\La(\mathbf{x}))\qquad \mathrm{for\ any}\ \mathbf{x}\in M.
\endaligned
\end{equation}
Combining \eqref{v-1nA}\eqref{3*lailaj*} and Lemma \ref{ISOtrans}, (compared with \eqref{vLa-1nA}) we have
\begin{equation}\aligned
\De_{F_\La(M)} v_{\La}^{-\f1n}\le\f{2+4\ep_n}{n(1-\ep_n)} v_{\La}^{-\f1n}|B_{F_\La(M)}|^2
\endaligned
\end{equation}
on $F_\La(M)$, where $B_{F_\La(M)}$ is the second fundamental form of $F_\La(M)$.
From Theorem \ref{MVIsuper} and Corollary \ref{CurvEst}, there are constants $\de_n\in(0,1]$ and $\th_{n}>0$ depending only on $n$ such that
\begin{equation}\aligned\label{FLathnnvLa*}
\f1{\mathcal{H}^n\left(F_\La(M)\cap \mathbf{B}_{r}(\mathbf{x})\right)}\int_{F_\La(M)\cap \mathbf{B}_{r}(\mathbf{x})}v_{\La}^{-\f{\de_n}n}\le \th_{n}v_{\La}^{-\f{\de_n}n}(\mathbf{x})
\endaligned
\end{equation}
for any $\mathbf{x}\in F_\La(M)$ with $d(\mathbf{x},F_\La(\p M))<2r$.

Let $L=\max_{B_1}|Du|$, then by the definition of $\bar{x}(M_r)$ in \eqref{barxMr**}, $\bar{x}(M_1)$ belongs to a ball $B_\r$ centered at the origin with radius $\r$ satisfying
\begin{equation}\aligned
\r\le\f{2\La+L}{\sqrt{4\La^2+1}}\le 1+\f{L}{2\La}.
\endaligned
\end{equation}
Hence, with \eqref{q1q2phiq1phiq2} there is a constant $c_n>0$ depending only on $n$ such that the number of the discrete set $\mathbf{Q}$ satisfies
\begin{equation}\aligned\label{sharpQ}
\sharp \mathbf{Q}\le c_n(1+L)^n.
\endaligned
\end{equation}
Up to a choice of $c_n$,
for any $\mathbf{q}\in \mathbf{Q}$ there is a finite sequence of points $\mathbf{q}_0,\mathbf{q}_1,\cdots,\mathbf{q}_m\in \mathbf{Q}$ with $\mathbf{q}_0=\mathbf{0}\in\R^n\times\R^n$, $\mathbf{q}=\mathbf{q}_m$, $m+1\le c_n(1+L)^n$ and $|\mathbf{q}_{i+1}-\mathbf{q}_i|\le1$ for $i=0,1,\cdots,m-1$.
Then from \eqref{FLathnnvLa*}
\begin{equation}\aligned\label{vLaf14f32}
\int_{F_\La(M)\cap \mathbf{B}_{\f14}(\mathbf{q}_1)}v_\La^{-\f{\de_n}n}\le\int_{F_\La(M)\cap \mathbf{B}_{\f32}(\mathbf{0})}v_\La^{-\f{\de_n}n}\le \th_nv_\La^{-\f{\de_n}n}(\mathbf{0})\mathcal{H}^n\left(F_\La(M)\cap \mathbf{B}_{\f32}(\mathbf{0})\right).
\endaligned
\end{equation}
For any $\mathbf{x},\mathbf{x'}\subset F_\La(\overline{M_1})$, from \eqref{La111222333} there is a constant $\a_n>0$ depending only on $n$ such that
\begin{equation}\aligned\label{xx'athn}
\mathcal{H}^n\left(F_\La(M)\cap \mathbf{B}_{\f32}(\mathbf{x'})\right)\le \f{\a_n}{\th_n}\mathcal{H}^n\left(F_\La(M)\cap \mathbf{B}_{\f14}(\mathbf{x})\right).
\endaligned
\end{equation}
Then from \eqref{vLaf14f32}, it follows that
\begin{equation}\aligned
\int_{F_\La(M)\cap \mathbf{B}_{\f14}(\mathbf{q}_1)}v_\La^{-\f{\de_n}n}\le\a_nv_\La^{-\f{\de_n}n}(\mathbf{0})\mathcal{H}^n\left(F_\La(M)\cap \mathbf{B}_{\f14}(\mathbf{q}_1)\right).
\endaligned
\end{equation}
Hence there is a point $\mathbf{z}_1\in\mathbf{B}_{\f14}(\mathbf{q}_1)$ such that $v_\La^{-\f{\de_n}n}(\mathbf{z}_1) \le\a_nv_\La^{-\f{\de_n}n}(\mathbf{0})$.

By an induction argument, we suppose that there is a point $\mathbf{z}_k\in\mathbf{B}_{\f14}(\mathbf{q}_k)$ such that $v_\La^{-\f{\de_n}n}(\mathbf{z}_k) \le\a_n^k v_\La^{-\f{\de_n}n}(\mathbf{0})$.
Since $B_{\f32}(\mathbf{z}_k)\cap F_\La(\p M)=\emptyset$, with \eqref{vLaf14f32}\eqref{xx'athn} we have
\begin{equation}\aligned
\int_{F_\La(M)\cap \mathbf{B}_{\f14}(\mathbf{q}_{k+1})}v_\La^{-\f{\de_n}n}\le&\int_{F_\La(M)\cap \mathbf{B}_{\f32}(\mathbf{z}_k)}v_\La^{-\f{\de_n}n}\le \th_nv_\La^{-\f{\de_n}n}(\mathbf{z}_k)\mathcal{H}^n\left(F_\La(M)\cap \mathbf{B}_{\f32}(\mathbf{z}_k)\right)\\
\le& \th_n\a_n^{k} v_\La^{-\f{\de_n}n}(\mathbf{0})\mathcal{H}^n\left(F_\La(M)\cap \mathbf{B}_{\f32}(\mathbf{z}_k)\right)\\
\le&\a_n^{k+1}v_\La^{-\f{\de_n}n}(\mathbf{0})\mathcal{H}^n\left(F_\La(M)\cap \mathbf{B}_{\f14}(\mathbf{q}_{k+1})\right).
\endaligned
\end{equation}
Hence there is a point $\mathbf{z}_{k+1}\in\mathbf{B}_{\f14}(\mathbf{q}_{k+1})$ such that $v_\La^{-\f{\de_n}n}(\mathbf{z}_{k+1}) \le\a_n^{k+1}v_\La^{-\f{\de_n}n}(\mathbf{0})$.
Therefore, with \eqref{La111222333}
\begin{equation}\aligned
\int_{F_\La(M)\cap \mathbf{B}_{1}(\mathbf{q})}v_\La^{-\f{\de_n}n}\le& \int_{F_\La(M)\cap \mathbf{B}_{\f32}(\mathbf{z}_m)}v_\La^{-\f{\de_n}n}\le \th_nv_\La^{-\f{\de_n}n}(\mathbf{z}_m)\mathcal{H}^n\left(F_\La(M)\cap \mathbf{B}_{\f32}(\mathbf{z}_m)\right)\\
\le&\a_n^{m+1}v_\La^{-\f{\de_n}n}(\mathbf{0})\mathcal{H}^n\left(F_\La(M)\cap \mathbf{B}_{\f14}(\mathbf{q})\right)\le c_n'\a_n^{c_n(1+L)^n}v_\La^{-\f{\de_n}n}(\mathbf{0}).
\endaligned
\end{equation}
Here, $c_n'$ is a constant depending only on $n$.
Combining \eqref{FMcover}\eqref{sharpQ}, we have
\begin{equation}\aligned
&\int_{F_\La(M_1)}v_\La^{-\f{\de_n}n}\le\sum_{\mathbf{q}\in \mathbf{Q}}\int_{F_\La(M)\cap \mathbf{B}_{1}(\mathbf{q})}v_\La^{-\f{\de_n}n}\\
\le& c_n'\a_n^{c_n(1+L)^n}v_\La^{-\f{\de_n}n}(\mathbf{0})c_n(1+L)^n=c_nc_n'e^{n\log(1+L)}e^{c_n(1+L)^n\log\a_n}v_\La^{-\f{\de_n}n}(\mathbf{0}).
\endaligned
\end{equation}
Therefore, from \eqref{vkvkLa*} we get
\begin{equation}\aligned
\int_{M_1}v^{-\f{\de_n}n}\le C_ne^{C_nL^n}v^{-\f{\de_n}n}(0)
\endaligned
\end{equation}
for some constant $C_n>0$ depending only on $n$. Then
\begin{equation}\aligned
\omega_n=\int_{M_1}v^{-1}\le\int_{M_1}v^{-\f{\de_n}n}\le C_ne^{C_nL^n}v^{-\f{\de_n}n}(0),
\endaligned
\end{equation}
which implies
\begin{equation}\aligned\label{v0CnLn}
v(0)\le\left(\f{C_n}{\omega_n}\right)^{\f n{\de_n}}e^{\f{nC_n}{\de_n}L^n}.
\endaligned
\end{equation}
This completes the proof.
\end{proof}
\begin{remark}
The order $n$ of $L^n$ in \eqref{v0CnLn} comes from the volume growth of $M$ (cite Proposition \ref{vol}).
In other words, if $\mathcal{H}^n(M_1)\le c(1+L)^\a$ for some $\a\in(0,n]$ with $M_1=\{(x,Du)\in\R^n\times\R^n|\ x\in B_1\}$ and $L=\max_{B_1}|Du|$, then we can improve the estimates in \eqref{sharpQ} to
$\sharp \mathbf{Q}\le c'(1+L)^\a$ for some constant $c'>0$ depending only on $n,c$. Correspondingly, \eqref{v0CnLn} can be improved to $v(0)\le c''e^{c''L^\a}$ for some constant $c''>0$ depending only on $n,c,\a$.
\end{remark}

Let $u$ be a smooth solution to the special Lagrangian equation \eqref{SL} on $B_R\subset\R^n$ with the eigenvalues $\la_1\ge\cdots\ge\la_n$ of $D^2u$. Assume $\Th>(n-2)\pi/2$.
Then
\begin{equation}\aligned
\Th=\sum_i\arctan\la_i<(n-1)\f{\pi}2+\arctan\la_n,
\endaligned
\end{equation}
which implies
\begin{equation}\aligned
\arctan(-\la_n)<(n-1)\f{\pi}2-\Th.
\endaligned
\end{equation}
Monotonicity of the function '$\arctan$' on $(-\f{\pi}2,\f{\pi}2)$ infers
\begin{equation}\aligned
-\la_n<\tan\left(\f{\pi}2-\left(\Th-(n-2)\f{\pi}2\right)\right)=\cot\left(\Th-(n-2)\f{\pi}2\right).
\endaligned
\end{equation}
Namely,
\begin{equation}\aligned
D^2u>-\cot\left(\Th-(n-2)\f{\pi}2\right).
\endaligned
\end{equation}
Let $\overline{\la}_u$ denote the largest eigenvalue of $D^2u$, and $\Phi=\sqrt{1+\overline{\la}_u^2}$ on $B_R$. Let $M$ be the special Lagrangian graph of $Du$ with Laplacian $\De_M$. In Proposition 2.1 of \cite{WdY}, Wang-Yuan proved
\begin{equation}\aligned\label{DeMb}
\De_M\log\Phi\ge\left(1-\f{4}{\sqrt{4n+1}+1}\right)|\na \log\Phi|^2
\endaligned
\end{equation}
in the distribution sense. Hence, $\De_M\Phi^{-\left(1-\f{4}{\sqrt{4n+1}+1}\right)}\le0$ in the distribution sense from \eqref{DeMb}.
From Theorem \ref{MVIsuper}, there are constants $\de_*\in(0,1]$ depending on $n$ and $\th_{n,\Th}>0$ depending only on $n,\max\{0,\cot\left(\Th-(n-2)\f{\pi}2\right)\}$ such that
\begin{equation}\aligned
\f1{\mathcal{H}^n\left(M\cap \mathbf{B}_r(z)\right)}\int_{M\cap \mathbf{B}_r(z)}\Phi^{-\de_*}\le \th_{n,\Th} \Phi^{-\de_*}(z)
\endaligned
\end{equation}
for any $z\in M$ and $0<r<\th_{n,\Th}^{-1}d(z,\p M)$.
Analog to the proof of Theorem \ref{GraSL}, we have the following result.
\begin{theorem}\label{GEsptP}
Let $u$ be a smooth solution to the special Lagrangian equation \eqref{SL} on $B_R\subset\R^n$ with $|\Th|>(n-2)\pi/2$.
Then there is a constant $C_{n,\Th}>0$ depending only on $n,\Th$ with $C_{n,\Th}\rightarrow\infty$ as $|\Th|\rightarrow(n-2)\pi/2$ such that
\begin{equation}\aligned
|D^2u(0)|\le C_{n,\Th} \mathrm{exp}\left(C_{n,\Th}\f{\max_{B_R}|Du-Du(0)|^n}{R^n}\right).
\endaligned
\end{equation}
\end{theorem}

\bibliographystyle{amsplain}

\end{document}